\documentclass[12pt]{article}
\usepackage{amsmath}
\usepackage{amssymb}
\usepackage{amsthm}
\usepackage{euscript}
\usepackage{graphicx}
\usepackage{rotating} 
\usepackage{hyperref}

\title{Fractional Iteration of Series and Transseries}

\author{G. A. Edgar}

\date{\today}

\theoremstyle{plain}
\newtheorem{pr}{Proposition}
\newtheorem{thm}[pr]{Theorem}
\newtheorem{co}[pr]{Corollary}
\newtheorem{lem}[pr]{Lemma}
\theoremstyle{remark}
\newtheorem{re}[pr]{Remark}
\newtheorem{de}[pr]{Definition}
\newtheorem{no}[pr]{Notation}
\newtheorem{ex}[pr]{Example}

\newtheorem{qu}[pr]{Question}
\numberwithin{pr}{section}

\newcommand{\Def}[1]{\textbf{\itshape #1}} 

\renewcommand{\epsilon}{\varepsilon}

\newcommand{\takes}{\colon}
\newcommand{\fgt}{\succ}
\newcommand{\fst}{\prec}
\newcommand{\fgteq}{\succcurlyeq}
\newcommand{\fsteq}{\preccurlyeq}
\newcommand{\fe}{\asymp}
\newcommand{\fgg}{\succ\!\!\!\succ}

\newcommand{\SET}[2]{ \left\{\, {#1} : {#2} \,\right\} }
\newcommand{\R}{\mathbb R}

\newcommand{\N}{\mathbb N}
\newcommand{\Z}{\mathbb Z}

\newcommand{\G}{\mathfrak G}

\newcommand{\GRID}{\mathfrak J}
\renewcommand{\AA}{\mathfrak A}
\newcommand{\BB}{\mathfrak B}

\newcommand{\DD}{\mathfrak D}
\newcommand{\T}{\mathbb T}
\newcommand{\LP}{\EuScript P}

\newcommand{\SA}{\EuScript A}

\newcommand{\bk}{\mathbf k}

\newcommand{\bm}{\mathbf m}
\newcommand{\0}{\mathbf 0}
\newcommand{\fa}{\mathfrak a}
\newcommand{\fb}{\mathfrak b}
\newcommand{\g}{\mathfrak g}
\newcommand{\m}{\mathfrak m}
\newcommand{\n}{\mathfrak n}
\newcommand{\e}{\mathfrak e}
\newcommand{\p}{\mathfrak p}
\renewcommand{\o}{\mathrm o}
\renewcommand{\O}{\mathrm O}

\newcommand{\Gsmall}{\G^{\mathrm{small}} }

\newcommand{\supp}{\operatorname{supp}}

\renewcommand{\mag}{\operatorname{mag}}
\newcommand{\dom}{\operatorname{dom}}
\newcommand{\expo}{\operatorname{expo}}
\newcommand{\rotomegabar}{{\phantom{\omega}}^{\begin{rotate}{180}%
$\scriptstyle\underline{\omega}$\end{rotate}}}
\newcommand{\acosh}{\operatorname{acosh}}

\newcommand{\bmu}{{\boldsymbol{\mu}}}
\newcommand{\ebmu}{{\boldsymbol{\mu}}}

\newcommand{\ba}{{\boldsymbol{\alpha}}}

\newcommand{\lbb}{\begin{picture}(8,8)(-2,2)
	\put(0,0){\line(0,1){9}}
	\put(2,0){\line(0,1){9}}
	\put(0,0){\line(1,0){5}}
	\put(0,9){\line(1,0){5}}
	\end{picture}}
\newcommand{\rbb}{\begin{picture}(7,8)(0,2)
	\put(3,0){\line(0,1){9}}
	\put(5,0){\line(0,1){9}}
	\put(0,0){\line(1,0){5}}
	\put(0,9){\line(1,0){5}}
	\end{picture}}
\newcommand{\lbbb}{\begin{picture}(10,8)(-2,2)
	\put(0,0){\line(0,1){9}}
	\put(2,0){\line(0,1){9}}
	\put(4,0){\line(0,1){9}}
	\put(0,0){\line(1,0){7}}
	\put(0,9){\line(1,0){7}}
	\end{picture}}
\newcommand{\rbbb}{\begin{picture}(9,8)(0,2)
	\put(3,0){\line(0,1){9}}
	\put(5,0){\line(0,1){9}}
	\put(7,0){\line(0,1){9}}
	\put(0,0){\line(1,0){7}}
	\put(0,9){\line(1,0){7}}
	\end{picture}}

\newcommand{\Bwellproduct}{Prop.~3.27}
\newcommand{\Bgridprod}{Prop.~3.35(d)}
\newcommand{\Baddendumsmall}{Prop.~3.52}
\newcommand{\Bheightwins}{Prop.~3.72}

\newcommand{\BcompoNM}{Prop.~3.111}

\newcommand{\Cwell}{Def.~2.1}
\newcommand{\Cwomonoid}{Prop.~2.2}

\newcommand{\Cwellderiv}{Prop.~2.5}
\newcommand{\Cuii}{Prop.~2.10}
\newcommand{\Cexponentiality}{Prop.~4.5}
\newcommand{\Cklargelogfree}{Prop.~4.8}

\newcommand{\Cinversei}{Prop.~4.19}
\newcommand{\Cinverse}{Prop.~4.20}
\newcommand{\Cmvtiii}{Prop.~4.23}
\newcommand{\Cmvtii}{Prop.~4.24}

\newcommand{\Wheredsubgrid}{Prop.~2.21}
\newcommand{\Wgridderiv}{Rem.~4.6}
\newcommand{\Wcompomono}{Prop.~5.6}
\newcommand{\Winv}{Sec.~8}

\begin{document} 
\maketitle
\setcounter{tocdepth}{2}
\allowdisplaybreaks
\abstract{We investigate compositional iteration of fractional
order for transseries.  For any large positive transseries $T$
of exponentiality $0$,
there is a family $T^{[s]}$ indexed by real numbers $s$
corresponding to iteration of order $s$.
It is based on Abel's Equation.
We also investigate the question of whether there is
a family $T^{[s]}$ all sharing a single support set.
A subset of the transseries of exponentiality $0$ is
divided into three classes (``shallow'', ``moderate'' and ``deep'')
with different properties related to fractional iteration.
}



\section*{Introduction}
Since at least as long ago as 1860
(A.~Cayley \cite{cayley})
there has been discussion of real iteration groups (compositional
iteration of fractional order) for power series.  Or at least
for formal power series, where we do not worry about
convergence of the result.  In this paper we adapt
this to \Def{transseries}.  In many cases it is, in fact, not
difficult to do when we ignore questions of convergence.

We will primarily use the ordered differential field
$\T = \R\lbb\G\rbb = \R\lbbb x \rbbb$ of
(real grid-based) transseries; $\T$ is also
known as \Def{the transline}.  So $\T$ is
the set of all grid-based real formal linear
combinations of monomials from $\G$, while $\G$
is the set of all $e^L$ for $L \in \T$ purely large.
(Because of logarithms, there is no need to write
separately two factors as $x^b e^L$.)
See ``Review'' below.

The problem looks like this:  Let $T$ be a large positive
transseries.  Is there a family $T^{[s]}$ of transseries,
indexed by reals $s$, so that:
$T^{[0]}(x)=x$, $T^{[1]}=T$, and
$T^{[s]}\circ T^{[t]} = T^{[s+t]}$ for all $s,t \in \R$?
These would be called \Def{fractional iterates} of $T$:
$U = T^{[1/2]}$ satisfies $U \circ U = T$; or
$T^{[-1]}$ is the compositional inverse of $T$; etc.
We limit this discussion to \emph{large positive}
transseries since that is where compositions
$S \circ T$ are always defined.
In Corollary~\ref{realiter} we conclude (in
the well-based case) that any
large positive $T$ of exponentiality $0$ admits such
a family of fractional iterates.
However, there are grid-based large positive transseries
$T$ and reals $s$ for which the fractional iterate $T^{[s]}$
is not grid-based (Example~\ref{not_grid}).

We also investigate the existence of a family $T^{[s]}, s \in \R$,
all supported by a single grid (in the grid-based case) or
by a single well-ordered set (in the well-based case).
We show that such a family exists in certain cases
(Theorems \ref{thmsubcrit} and~\ref{thmcrit}) but not
in other cases (Theorem~\ref{thm:deep}).

\subsection*{Review}
The differential field $\T$ of transseries is completely explained in
my recent expository introduction \cite{edgar}.  Other
sources for the definitions are:
\cite{asch}, \cite{costintop}, \cite{DMM}, \cite{hoeven}.
I will generally follow
the notation from \cite{edgar}.
The well-based version of the construction is described in
\cite{DMM} (or \cite[\Cwell]{edgarc}).  In this paper it is intended that
all results hold for both versions, unless otherwise noted.
(I use labels \fbox{G} or \fbox{W} for
statements or proofs valid only for the grid-based
or well-based version, respectively.)
We will write $\T$, $\G$, and so on in both versions.

Write $\LP = \SET{S \in \T}{S \fgt 1, S > 0}$ for the set of
large positive transseries. The operation of
composition $T \circ S$ is defined for $T \in \T$, $S \in \LP$.
The set $\LP$ is a group under composition
(\cite[\S~5.4.1]{hoeven}, \cite[Cor.~6.25]{DMM},
\cite[\Cinverse]{edgarc}, \cite[\Winv]{edgarw}).
Both notations $T\circ S$ and $T(S)$ will be used.

We write $\G$ for the ordered abelian group of transmonomials.
We write $\G_{N,M}$ for the transmonomials with
exponential height $N$ and logarithmic depth $M$.
We write $\G_N$ for the log-free transmonomials with height $N$.
Even in the well-based case, the definition is
restricted so that for any $T \in \T$
there exists $N,M$ with $\supp T \subseteq \G_{N,M}$.
The support $\supp T$ is well ordered for the converse of
the relation $\fst$ in the well-based case;
the support $\supp T$ is a subgrid in the grid-based case.
A ratio set $\bmu$ is a finite subset of $\Gsmall$;
$\GRID^\ebmu$ is the group generated by $\bmu$.
If $\bmu = \{\mu_1,\cdots,\mu_n\}$, then
$\GRID^\ebmu = \SET{\bmu^\bk}{\bk \in \Z^n}$.
If $\bm \in \Z^n$, then
$\GRID^{\ebmu,\bm} = \SET{\bmu^\bk}{\bk \in \Z^n, \bk \ge \bm}$
is a grid.  A grid-based transseries is supported
by some grid.  A subgrid is a subset of a grid.

For transseries $A$, we already use
exponents $A^n$ for multiplicative powers,
and parentheses $A^{(n)}$ for derivatives.  Therefore
let us use square brackets
$A^{[n]}$ for compositional powers.
In particular, we will write $A^{[-1]}$ for the compositional
inverse.  Thus, for example, $\exp_n = \exp^{[n]} = \log^{[-n]}$.

In this paper we will be using some of the results on
composition of transseries from \cite{edgarc} and \cite{edgarw},
where the proofs are sometimes not as simple as in
\cite{edgar} (and not all of them are proved there).

\section{Motivation for Fractional Iteration}
Before we turn to transseries, let us consider
fractional iteration in general.
Given functions $T \takes X \to X$ and
$\Phi \takes \R \times X \to X$, we say that
$\Phi$ is a \Def{real iteration group} for $T$ iff
\begin{align}
	\Phi(s+t,x) &= \Phi\big(s,\Phi(t,x)\big),\label{eq:defn}
	\\
	\qquad \Phi(0,x) &= x,\label{eq:0}
	\\
	\Phi(1,x) &= T(x) ,\label{eq:1}
\end{align}
for all $s,t \in \R$ and $x \in X$.

Let us assume, say, that
$X$ is an interval $(a,b) \subseteq \R$, possibly $a=-\infty$
and/or $b=\infty$,
and that $\Phi$ has as many derivatives
as needed.
If we start with (\ref{eq:defn}), take the partial derivative with
respect to $t$,
$$
	\Phi_1(s+t,x) = \Phi_2\big(s,\Phi(t,x)\big)\,\Phi_1(t,x) ,
$$
then substitute $t=0$, we get
\begin{equation}
	\Phi_1(s,x) = \Phi_2(s,x)\,\Phi_1(0,x) .\label{eq:pde}
\end{equation}
We have written $\Phi_1$ and $\Phi_2$ for the two partial
derivatives of $\Phi$.  Equation (\ref{eq:pde}) is
the one we will be using in Sections \ref{s:3ex}
and~\ref{s:lpt} below.  Here is a proof
showing how it works in the case of functions
on an interval $X$.

\begin{pr}
Suppose $\Phi \takes \R \times X \to X$ satisfies
{\rm(\ref{eq:0})} and
\begin{equation*}
	\Phi_1(s,x) = \Phi_2(s,x)\,\beta(x)
\tag{\ref{eq:pde}${}'$}
\end{equation*}
with $\beta(x) > 0$.  {\rm[}Assume also
$\infty = \int_{x_0}^b dy/\beta(y)$ and
$\infty = \int_a^{x_0} dy/\beta(y)$ for $x_0 \in X = (a,b)$.{\rm]}
Then $\Phi$ satisfies
{\rm (\ref{eq:defn})} with
$\beta(x) = \Phi_1(0,x)$.
\end{pr}
\begin{proof}
Fix $x_0 \in X$.  Let $\theta(t)$ be the solution of the ODE
$\theta'(t) = \beta(\theta(t))$, $\theta(0) = x_0$.
That is, $\theta(t)= x$ is defined implicitly by
$$
	\int_{x_0}^x \frac{dy}{\beta(y)} = t .
$$
[In order to get all time $t$, we need
$\infty = \int_{x_0}^b dy/\beta(y)$ and
$\infty = \int_a^{x_0} dy/\beta(y)$.]
Consider $F(u,v) = \Phi(u+v,\theta(u-v))$.  Then
\begin{align*}
	F_2(u,v) &= \Phi_1\big(u+v,\theta(u-v)\big)
	- \Phi_2\big(u+v,\theta(u-v)\big)\,\theta'(u-v)
	\\ &
	= \Phi_1\big(u+v,\theta(u-v)\big)
	- \Phi_2\big(u+v,\theta(u-v)\big)\,\beta\big(\theta(u-v)\big)
	= 0
\end{align*}
by (\ref{eq:pde}${}'$).  This means $F$ is independent of $v$.
So
\begin{align*}
	\theta(t) &= \Phi(0,\theta(t)) = F(t/2,-t/2) = F(t/2,t/2)
	\\ &= \Phi(t,\theta(0)) = \Phi(t,x_0),
	\\
	\Phi(s,\Phi(t,x_0)) &= \Phi(s,\theta(t)) = F((s+t)/2,(s-t)/2)
	\\&=F((s+t)/2,(s+t)/2)
	= \Phi(s+t,\theta(0)) = \Phi(s+t,x_0) .
\end{align*}
Differentiate $\Phi(t,x_0)= \theta(t)$ to get
$\Phi_1(t,x_0) = \theta'(t) = \beta(\theta(t))$,
then substitute $t=0$ to get
$\Phi_1(0,x_0) = \beta(\theta(0)) = \beta(x_0)$.
\end{proof}

\section{Three Examples}\label{s:3ex}
\subsection*{Power Series}
We start with the classical case of power series.
(A.~Cayley 1860 \cite{cayley}; A.~Korkine 1882 \cite{korkine}.)
We will think of formal power series (for $x \to \infty$),
and not actual functions.
Consider a series of the form
\begin{align}\label{eq:pseries}
	T(x) &= x\left(1+\sum_{j=1}^\infty c_j x^{-j}\right)
	\\
	&= x + \sum_{j=1}^\infty c_j x^{-j+1}
	= x + c_1 + c_2 x^{-1} + c_3 x^{-2} +\cdots
\notag
\end{align}
Such a series
admits an iteration group of the same form.  That is,
\begin{equation}\label{eq:pseriesa}
	\Phi(s,x) = x\left(1+\sum_{j=1}^\infty \alpha_j(s) x^{-j}\right).
\end{equation}
In fact, $\alpha_j(s)$ is
$s c_j + {}$ \{polynomial in $s, c_1, c_2, \dots, c_{j-1}$
with rational coefficients, of degree $j-1$ in $s$\}.
The first few terms:
\begin{align*}
	\Phi(s,x) &= x + sc_1 + sc_2 x^{-1} +
	\left(sc_3 + \frac{s(1-s)}{2}c_1c_2\right)x^{-2}
	\\ &+
	\left(sc_4+\frac{s(1-s)}{2}(2c_1c_3+c_2^2)
	+\frac{s(1-s)(1-2s)}{6}c_1^2c_2\right)x^{-3} +\cdots .
\end{align*}

\begin{thm}\label{th:powseries}
Let $T(x)$ be the power series
{\rm(\ref{eq:pseries})}.  Define
$\alpha_j \takes \R \to \R$ recursively by
\begin{align*}
	\alpha_1(s) &= s c_1,
	\\
	\alpha_j(s) &= s\left(c_j-\int_0^1 \sum_{j_1+j_2=j}
	(-j_1+1) \alpha_{j_1}(u) \alpha_{j_2}'(0)\,du\right)
	\\&\qquad+ \int_0^s \sum_{j_1+j_2=j}
	(-j_1+1) \alpha_{j_1}(u) \alpha_{j_2}'(0)\,du .
\end{align*}
Then the series $\Phi$ defined formally by {\rm(\ref{eq:pseriesa})}
is a real iteration group for $T$.
\end{thm}
\begin{proof}
Check (\ref{eq:0}), (\ref{eq:1}), (\ref{eq:pde}).
\end{proof}

\begin{re}
The formulas are obtained by plugging (\ref{eq:pseriesa})
into (\ref{eq:pde}), equating coefficients, then integrating
the resulting ODEs.  Consequently, this is the \emph{unique}
solution of the form (\ref{eq:pseriesa}) with differentiable
coefficients.
\end{re}

\begin{re}
Of course there is a corresponding formulation for series
of the form
$$
	T(z) = z \left(1+\sum_{j=1}^\infty c_j z^j\right)
	= z + c_1 z^2 + c_2 z^3 + \cdots .
$$
Then we get $\Phi(s,z) = z\big(1+\sum_{j=1}^\infty \alpha_j(s) z^j\big)$,
with
\begin{align*}
	\alpha_1(s) &= s c_1,
	\\
	\alpha_j(s) &= s\left(c_j - \int_0^1 \sum_{j_1+j_2=j}(j_1+1)
	\alpha_{j_1}(u) \alpha_{j_2}'(0)\,du\right)
	\\ &\qquad+ \int_0^s \sum_{j_1+j_2=j}(j_1+1)
	\alpha_{j_1}(u) \alpha_{j_2}'(0)\,du .
\end{align*}
The first few terms are:
\begin{align*}
	\Phi(s,z) &= z + s c_1 z^2 + \left(s c_2 + s(s-1) c_1^2\right)z^3
	\\ &\qquad+\left(sc_3+\frac{5s(s-1)}{2}c_1c_2+
	\frac{s(s-1)(2s-3)}{2}c_1^3\right)z^4
	+\cdots .
\end{align*}
\end{re}

\subsection*{Convergence}
We considered here \emph{formal} series.  Even if
(\ref{eq:pseries}) converges for all $x$ (except $0$),
it need not follow that (\ref{eq:pseriesa}) converges.
Indeed, one of the criticisms of Cayley \cite{cayley}
and Korkine \cite{korkine} was that convergence
was not proved.  Baker \cite{baker} provides examples
where a power series converges, but none of its non-integer
iterates converges.  Erd\"os \& Jabotinsky \cite{erdos}
investigate the set of $s$ for which the series converges.

\subsection*{Transseries, Height and Depth 0}
Let $B \subseteq (0,\infty)$ be a well ordered set (under the
usual order).  The transseries
\begin{equation}\label{eq:lev0}
	T(x) = x\left(1 + \sum_{b \in B} c_b x^{-b}\right)
	= x + \sum_{b \in B} c_b x^{-b+1}
\end{equation}
is the next one we consider.  In fact, the additive semigroup
generated by $B$ is again well ordered (Higman,
see the proof in \cite[\Cwomonoid]{edgarc}), so we will
assume from the start that $B$ is a semigroup.
If $B$ is finitely generated, then (\ref{eq:lev0})
is a grid-based transseries.  But in general it is well-based.
If $B$ is finitely generated, then $B$ has
order type $\omega$.  But of course
a well ordered $B$ can have arbitrarily large
countable ordinal as order type.

We claim $T$ has a real iteration group
supported by the same $1-B$, where $B$ is a well-ordered additive semigroup
of positive reals.

\begin{pr}\label{sub}
Let $B \subseteq (0,\infty)$ be a semigroup.  The transseries
{\rm(\ref{eq:lev0})} has a real iteration group
$$
	\Phi(s,x) = x\left(1+\sum_{b \in B} \alpha_b(s) x^{-b}\right) .
$$
\end{pr}
\begin{proof}
The only thing needed is that $B$ is a well ordered semigroup.
It follows that, for any given $b \in B$, there are just
finitely many pairs $(b_1,b_2) \in B \times B$ with
$b_1+b_2 = b$ \cite[\Bwellproduct]{edgar}.  Then define recursively:
\begin{align*}
	f_b(s) &= \sum_{b_1+b_2=b} (-b_1+1) \alpha_{b_1}(s) \alpha'_{b_2}(0),
	\\
	\alpha_b(s) &= s\left(c_b-\int_0^1 f_b(u)\,du\right)
	+ \int_0^s f_b(u)\,du .
\end{align*}
For each $b$, both $f_b(s)$ and
$\alpha_b(s)$ are polynomials (finitely
many terms!) in $s$ and the $c_{b_1}$ [with $b_1 < b$
except for the term $sc_b$].
Check (\ref{eq:0}), (\ref{eq:1}), (\ref{eq:pde}).
\end{proof}

\subsection*{A Moderate Example}
Now we consider another case.
We single it out because it occurs frequently enough to make it useful
to have the formulas displayed.
Consider the transseries
\begin{equation}\label{eq:height1}
	T(x) =
	x\left(\sum_{k=0}^\infty \sum_{j=0}^\infty c_{j,k}x^{-j}
	e^{-kx}\right), \qquad c_{0,0} = 1 .
\end{equation}
The set
\begin{equation}\label{eq:setB}
	B = \SET{(j,k)}{k \ge 0, j \ge 0, (j,k) \ne (0,0)}
\end{equation}
is a semigroup under addition.  The set
$\SET{x^{-j} e^{-kx}}{(j,k) \in B}$ is then a semigroup
under multiplication.  It is well ordered with order
type $\omega^2$ with respect to the converse of $\fgt$.

\begin{thm}\label{critical}
Let $B$ be as in {\rm(\ref{eq:setB})}.  Then the
transseries {\rm(\ref{eq:height1})} admits a real iteration group
supported by the same set $\SET{x^{1-j} e^{-kx}}{j,k \ge 0}$.
\end{thm}
\begin{proof}
Write
\begin{equation*}
	\Phi(s,x) = x\left(1+\sum_{(j,k) \in B}
	\alpha_{j,k}(s) x^{-j} e^{-kx}\right) .
\end{equation*}

(A) We first consider the case with $c_{1,0} = 0$.
The coefficient functions $\alpha_{j,k}$ are defined
recursively as follows.
If $(j,k) \notin B$, then let $\alpha_{j,k}(s) = 0$.
Let $(j,k) \in B$ and assume
$\alpha_{j_1,k_1}(s)$ have already been defined for all
$(j_1,k_1)$ with either $k_1 < k$ or
\{$k_1=k$ and $j_1 < j$\}.  Then let
\begin{equation*}
	f_{j,k}(s) =
	\sum
	\big[(-j_1+1)\alpha'_{j_2,k_2}(0) -k_1 \alpha'_{j_2+1,k_2}(0)\big]
	\alpha_{j_1,k_1}(s) ,
\end{equation*}
where the sum is over all $j_1,j_2,k_1,k_2$ with
$j_1+j_2=j, k_1+k_2=k$.  Check that all the terms in the
sum involve $\alpha$'s that have already been defined (or are
multiplied by zero); this depends on $c_{1,0} = 0$,
so $\alpha_{1,0}(s) = 0$.
Define $F_{j,k}(s) = \int_0^s f_{j,k}(u)\,du$ and
\begin{equation*}
	\alpha_{j,k}(s) =
	\big(c_{j,k} - F_{j,k}(1)\big) s
	+F_{j,k}(s) .
\end{equation*}
Check (\ref{eq:0}), (\ref{eq:1}), (\ref{eq:pde}).

(B) Now consider the case with $c_{1,0} \ne 0$.  Write $c = c_{1,0}$.
The coefficient functions $\alpha_{j,k}$ are defined
recursively as follows.
If $(j,k) \notin B$, then let $\alpha_{j,k}(s) = 0$.
Let $(j,k) \in B$ and assume
$\alpha_{j_1,k_1}(s)$ have already been defined for all
$(j_1,k_1)$ with either $k_1 < k$ or
\{$k_1=k$ and $j_1 < j$\}.  Then let
\begin{equation*}
	f_{j,k}(s) =
	\sum
	\big[(-j_1+1)\alpha'_{j_2,k_2}(0) -k_1 \alpha'_{j_2+1,k_2}(0)\big]
	\alpha_{j_1,k_1}(s) ,
\end{equation*}
where the sum is over all $j_1,j_2,k_1,k_2$ with
$j_1+j_2=j, k_1+k_2=k$.  Omit the terms $\alpha_{jk}'(0)$
and $-kc\alpha_{jk}(s)$ and terms with a factor $0$.
Then all terms
in the sum involve $\alpha$'s already defined.
Define $F_{j,k}(s) = \int_0^s e^{kc(u-s)}f_{j,k}(u)\,du$ and
\begin{equation*}
	\alpha_{j,k}(s) =
	\begin{cases}
	\big(c_{j,0} - F_{j,0}(1)\big) s +F_{j,0}(s) ,& \text{if } k=0,
	\\ \null \\
	\big(c_{j,k} - F_{j,k}(1)\big)
	\frac{\displaystyle 1-e^{-sck}}{\displaystyle 1-e^{-ck}}
	+F_{j,k}(s) ,& \text{if } k>0 .
	\end{cases}
\end{equation*}
Recall $c = c_{1,0}$.
Check (\ref{eq:0}), (\ref{eq:1}), (\ref{eq:pde}).
\end{proof}

In case (B)---the ``moderate case''---the coefficients
$\alpha_{j,k}(s)$ are not necessarily
polynomials in $s$.

\subsection*{A Deep Example}
Another simple example shows that real iteration group
of that type
need not always exist.  Let $B \subseteq \Z^3$ be
\begin{equation}\label{eq:superB}
	B = \SET{(j,0,0)}{j \ge 1} \cup \SET{(j,k,0)}{k \ge 1}
	\cup \SET{(j,k,l)}{l \ge 1} .
\end{equation}
Then $\SET{x^{-j}e^{-kx}e^{-lx^2}}{(j,k,l) \in B}$
is a semigroup, but not well ordered.

I included some negative $j$ and $k$ so that
the set of transseries of the form
$$
	x\left(1+\sum_{(j,k,l) \in B} c_{jkl}
	x^{-j} e^{-kx} e^{-lx^2}\right)
$$
(where of course each individual support is well ordered, not all
of $B$) is closed under composition.

\begin{pr}\label{super}
Let $B'$ be a well ordered subset of {\rm(\ref{eq:superB})}.
The transseries $T(x) = x(1+x^{-1}+e^{-x^2})$ admits no
real iteration group of the form
$$
	\Phi(s,x) = x\left(1+\sum_{(j,k,l) \in B'} \alpha_{jkl}(s)
	x^{-j} e^{-kx} e^{-lx^2}\right) .
$$
\end{pr}
\begin{proof}
We may assume $B' = \SET{(j,k,l)}{\alpha_{jkl} \ne 0}$.
As before, the first term beyond $x$ can be computed
as $\alpha_{100}(s) = s \cdot 1 = s$, so $\alpha'_{100}(s) = 1$.
Now $\alpha_{001}(1) \ne 0$, so there is a least
$(j,k,1) \in B'$.  But then by considering the coefficient
of $x^{2-j}e^{-kx}e^{-x^2}$ in $\Phi_1(s,x) = \Phi_2(s,x) \Phi_1(0,x)$ we have
$$
	0 = (-2) \alpha_{jk1}(s)\alpha'_{100}(0),
$$
so $\alpha_{jk1}(s) = 0$, a contradiction.
\end{proof}

\begin{re}
An alternate argument that there is no grid (or
well ordered set) supporting a real iteration group for
$T = x + 1 + xe^{-x^2}$.  Compute
\begin{align*}
	\supp T^{[-1]} &= \{x \fgt 1 \fgt xe^{2x}e^{-x^2} \fgt \cdots\}
	\\
	\supp T^{[-2]} &= \{x \fgt 1 \fgt xe^{4x}e^{-x^2} \fgt \cdots\}
	\\
	\supp T^{[-3]} &= \{x \fgt 1 \fgt xe^{6x}e^{-x^2} \fgt \cdots\}
	\\
	& \dots
	\\
	\supp T^{[-k]} &= \{x \fgt 1 \fgt xe^{2kx}e^{-x^2} \fgt \cdots\},
	\qquad k \in \N, k > 0.
\end{align*}
So there is no grid (and no well ordered set) containing
all of these supports.
\end{re}

\section{The Case of Common Support}\label{s:lpt}
\subsection*{Conjugation}
The set $\LP$ of large positive transseries is a group
under composition (\cite[\Cinverse]{edgarc},
\cite[\Winv]{edgarw}).
In a group, we say $U, V$ are \Def{conjugate}
if there exists $S$ with $S^{[-1]} \circ U\circ S = V$.
Then for all $k \in \N$ it follows that
$S^{[-1]} \circ U^{[k]}\circ  S = V^{[k]}$.
If $\Phi(s,x)$ is an iteration group
for $U(x)$, then
$S^{[-1]}(\Phi(s,S(x)))$ is an iteration group
for the conjugate $S^{[-1]}(U(S(x)))$.  This
can be used to reduce the question of fractional
iteration for certain more general transseries
to more restricted cases to be discussed here.

\subsection*{Puiseux series}
For a Puiseux series of the form
\begin{equation}\label{eq:puiseux}
	T = \sum_{j=m}^\infty c_j x^{-j/k}
\end{equation}
($m \in \Z, k \in \N, c_m>0$), we can conjugate with $x^{1/k}$:
\begin{align*}
	x^{1/k} \circ T \circ x^{k}
	&= \left(T(x^k)\right)^{1/k}
	= \left(\sum_{j=m}^\infty c_j x^{-j}\right)^{1/k}
	\\
	&= c_m^{1/k} x^{-m/k} \left(1+c_{m+1}x^{-1}
	+c_{m+2}x^{-2}+\cdots\right)^{1/k}
	\\
	&= c_m^{1/k} x^{-m/k} \left(1+a_1x^{-1}
	+a_2x^{-2}+\cdots\right) .
\end{align*}
If $\dom T = x$, then also
$\dom(x^{1/k} \circ T \circ x^{k}) = x$ and
existence of a real iteration group is then clear from
Proposition~\ref{eq:pseries}.  (If $\dom T \ne x$,
keep reading.)

\subsection*{Exponentiality}
Associated to a general $T \in \LP$
is an integer $p$ called the \Def{exponentiality} of $T$
\cite[Ex.~4.10]{hoeven} \cite[\Cexponentiality]{edgarc}
such that for all large
enough $k \in \N$ we have
$\log_{k} \circ \;T \circ \exp_{k} \sim \exp_p$.
Write $p = \expo T$.

Now $\expo(S \circ T) = \expo S +\expo T $, so no transseries
with nonzero exponentiality can have a real iteration group of transseries.
There is no transseries $T$ with $T \circ T = e^x$.
(But see \cite{edgartet}.)  The
main question will be for exponentiality zero.  If $\expo T = 0$,
then $T$ is conjugate to some $S = \log_k \circ\;T\circ \exp_k$ such that
$S \sim x$ and such that
$S$ is log-free \cite[\Cklargelogfree]{edgarc}.
So we will deal with this case.

\subsection*{Shallow---Moderate---Deep}
Now we turn to the general large positive log-free transseries
with dominant term $x$.  It admits a unique real iteration group
with a common support in many cases
(shallow and moderate), but not in many other cases (deep).

\begin{de}
Consider log-free $T \sim x$.  A real iteration group
for $T$ with \Def{common support} is a real
iteration group $\Phi(s,x)$ of the form
\begin{equation}\label{eq:phiseries}
	\Phi(s,x) = x\left(1+\sum_{\g \in \BB} \alpha_\g(s) \g\right)
\end{equation}
for some subgrid (\,\fbox{W} or well ordered)
$\BB \subseteq \G$ (not depending on $s$)
where coefficient functions $\alpha_\g \takes \R \to \R$
are differentiable.
\end{de}

Write $T = x(1+U)$,
$U \fst 1$, $U \sim a\e$, $a \in \R$,
$a \ne 0$, $\e \in \G$, $\e \fst 1$.
As before, if there is a real interation group $\Phi$,
it begins $\Phi(s,x) = x(1 + sa\e + \cdots)$.
We may assume:
if $\g \in \BB$, then $\alpha_\g(s) \ne 0$ for some $s$.
So the greatest element of $\BB$ is $\e$.

Write $A^\dagger = A'/A$ for the logarithmic derivative.

\begin{de}\label{de:deep}
Let $T = x(1+U)$, $U \fst 1$,
$\mag U = \e$.  Monomial $\e$ is
called the \Def{first ratio} of $T$.
We say that $T$ is:

\Def{shallow} iff $\g^\dagger \fst 1/(x\e)$ for all
$\g \in \supp U$;

\Def{moderate} iff  $\g^\dagger \fsteq 1/(x\e)$  for all
$\g \in \supp U$
and $\g^\dagger \fe 1/(x\e)$ for at least one $\g \in \supp U$;

\Def{deep} iff $\g^\dagger \fgt 1/(x\e)$ for some
$\g \in \supp U$.

\Def{purely deep} iff $\g^\dagger \fgt 1/(x\e)$ for all
$\g \in \supp U$ except $\e$.
\end{de}

\begin{re}
\fbox{G}
It may be practical to check these definitions
using a ratio set $\bmu = \{\mu_1, \cdots, \mu_n\}$.
For example:
Suppose $\supp U \subseteq \GRID^\ebmu$.  By the group property
(Lemma \ref{lem:group}): if
$\mu_i^\dagger \fst 1/(x\e)$ for $1 \le i \le n$, then
$T$ is shallow.  This will be ``if and only if''
provided $\bmu$ is chosen from the group generated by $\supp U$,
which can always be done.
\end{re}

\begin{re}\label{ecase}
The case of two terms, $T = x(1 + a\e)$ exactly, is shallow.
Indeed, $\e \fst 1$ and so $x\e \fst x$,
$x\e' \fst 1$ and $\e'/\e \fst 1/(x\e)$.
\end{re}

\begin{re}
For small monomials, the logarithmic derivative operation
reverses the order:
if $1 \fgt \fa \fgt \fb$, then $\fa^\dagger \fsteq \fb^\dagger$
(Lemma~\ref{lem:logder}(e)).  And for $U \fst 1$ we have
$U^\dagger \sim (\mag U)^\dagger$ (Lemma~\ref{lem:logder}(c)).
So $T=x(1+ a\e + V)$, $V \fst \e$, is purely deep if and only if
$V^\dagger \fgt 1/(x\e)$.
\end{re}

\begin{re}
The condition $\g^\dagger \fst 1/(x\e)$ says that $\g$ is ``not too small''
in relation to $\e$. (This is the reason for the terms
``shallow'' and ``deep''.)  If $\g \fst 1$ then $\g = e^{-L}$ with
$L>0$ purely large and
\begin{align*}
	\g^\dagger \fst \frac{1}{x\e} & \Longleftrightarrow
	L' \fst \frac{1}{x\e} 
	\\ & \Longleftrightarrow
	L \fst \int\frac{1}{x\e}
	\\ & \Longleftrightarrow
	L < c\int\frac{1}{x\e}\;\text{for all real $c>0$}
	\\ & \Longleftrightarrow
	e^{L} < \exp\left(c\int\frac{1}{x\e}\right)\;\text{for all real $c>0$}
	\\ & \Longleftrightarrow
	\g > \exp\left(-c\int\frac{1}{x\e}\right)\;\text{for all real $c>0$} .
\end{align*}
So the set $\AA = \SET{\g \in \G}{\g \fsteq \e, \g^\dagger \fst 1/(x\e)}$
is an \Def{interval} in $\G$.  The large end
of the interval is the first ratio $\e$, the small end of the interval
is the gap in $\G$ just above all the values
$\exp(-c\int (1/{x\e}))$, $c \in \R$, $c > 0$.  If we
write $\exp\big({-}\overline{0}\int(1/{x\e})\big)$ for that gap, then
$$
	\AA = \big]\;{\textstyle\exp\big({-}\overline{0}\int(1/{x\e})\big)}
	, \e\;\big] .
$$
I will call this the \Def{shallow interval} below $\e$.
Van der Hoeven devotes a chapter \cite[Chap.~9]{hoeven} to gaps (cuts)
in the transline. In his classification \cite[Prop.~9.15]{hoeven}, 
$$
	{\textstyle\exp\big({-}\overline{0}\int(1/{x\e})\big)} =
	e^{\displaystyle -e^{\displaystyle A-e^{\rotomegabar}}},
$$
where $\mag\int(1/(x\e)) = e^A$.

Similarly, the set
$\AA = \SET{\g \in \G}{\g \fsteq \e, \g^\dagger \fsteq 1/(x\e)}$
is an interval in $\G$.  The large end
of the interval is the first ratio $\e$, the small end of the interval
is the gap in $\G$ just below all the values
$\exp\big({-}c\int (1/{x\e})\big)$, $c \in \R$, $c > 0$.  If we
write $\exp\big({-}\overline{\infty}\int(1/{x\e})\big)$ for that gap, then
$$
	\AA = \big]\;{\textstyle\exp\big({-}\overline{\infty}\int(1/{x\e})\big)}
	, \e\;\big] .
$$
I will call this the \Def{moderate interval} below $\e$.
In van der Hoeven's classification, 
$$
	{\textstyle\exp\big({-}\overline{\infty}\int(1/{x\e})\big)} =
	e^{\displaystyle -e^{\displaystyle A+e^{\rotomegabar}}},
$$
where $\mag\int(1/(x\e)) = e^A$.

\end{re}

\subsection*{The Examples}
Let us examine where the examples done above fit in the
shallow/deep classification.  If $\e = x^{-1}$, then
$1/(x\e) = 1$, $\int(1/(x\e)) = x$,
$\exp\big({-}c\int (1/(x\e))\big) = e^{-cx}$.  The small
end of the shallow interval is $\exp(-\overline{0}x)$.
In a power
series, every monomial $x^{-j} \fgt e^{-\overline{0}x}$
is inside the shallow interval,
so a power series is shallow.

In example (\ref{eq:height1}), we saw two cases.  In case
$c_{1,0} \ne 0$, then $\e = x^{-1}$ so again the
small end of the shallow interval is
$\exp(-\overline{0}x)$.  But the monomial
$x^{-j}e^{-kx} \fst \exp(-\overline{0}x)$ if $k > 0$,
and is thus outside the shallow interval,
so this is not shallow.  The small end of the moderate
interval is $\exp(-\overline{\infty}x)$, and all
monomials $x^{-j}e^{-kx} \fgt \exp(-\overline{\infty}x)$
are inside the moderate interval,
so this is the moderate case.

The other case is $c_{1,0} = 0$.  Then $\e$ is $x^{-2}$ (or smaller).
If $\e = x^{-2}$, then $1/(x\e)=x$,
$\exp\big({-}c\int(1/(x\e))\big)= e^{-(c/2)x^2}$.  The small
end of the shallow interval is $\exp(-\overline{0}x^2)$.
All monomials $x^{-j}e^{-kx} \fgt \exp(-\overline{0}x^2)$
are inside the shallow interval,
so this is the shallow case.

Finally consider the example $T = x(1 + x^{-1} + e^{-x^2})$
of Proposition~\ref{super}.  Since $\e = x^{-1}$, the small end
of the moderate interval was computed as $\exp(-\overline{\infty}x)$.
The monomial $e^{-x^2} \fst \exp(-\overline{\infty}x)$ is
outside of that, so $T$ is deep.

\subsection*{Proofs}Proofs will follow the examples done above.
These proofs use some
technical lemmas on logarithmic derivatives, grids,
and well ordered sets;
they are found after the main results, starting
with Lemma~\ref{lem:logder}.

\begin{thm}\label{thmsubcrit}
If (log-free) $T \sim x$ is shallow, then $T$ admits a real iteration
group with common support
where all coefficient functions are polynomials.
\end{thm}
\begin{proof}
The proof is as in Proposition~\ref{sub} above.  Here
are the details.  Write $T = x(1+U)$, $U \sim a\e$,
$\e \fst 1$.  Begin with the subgrid (\,\fbox{W} well ordered)
$\supp U$ which is contained in
$\SET{\g \in \G}{\g \fgteq \e, \g^\dagger \fst 1/(x\e)}$.
Let $\BB \supseteq \supp U$ be the least set such that
if $\g_1, \g_2 \in \BB$, then
$\supp\big((x\g_1)'\g_2\big) \subseteq \BB$.
By Lemma~\ref{lem:gensubgrid}, $\BB$ is a subgrid
(\,\fbox{W} by Lemma~\ref{lemw:gen}, $\BB$ is well ordered)
and $\BB \subseteq \SET{\g \in \G}{\g \fgteq \e,
\g^\dagger \fst 1/(x\e)}$.

Write 
\begin{align}\label{eq:pdecomp}
	T(x) &= x\left(1+\sum_{\g \in \BB} c_\g\g\right)
	\notag\\
	\Phi(s,x) &= x\left(1+\sum_{\g \in \BB} \alpha_\g(s)\g\right)
	\notag\\
	\Phi_1(s,x) &= x\sum_{\g \in \BB} \alpha_\g'(s)\g
	\\
	\Phi_1(0,x) &= x\sum_{\g \in \BB} \alpha_\g'(0)\g
	\notag\\
	\Phi_2(s,x) &= 1 + \sum_{\g \in \BB} \alpha_\g(s)\,(x\g)'
	\notag\\
	\Phi_2(s,x)\Phi_1(0,x) &= x\sum_{\g \in \BB} \alpha_\g'(0)\g
	+ x\sum_{\g_1,\g_2 \in \BB} \alpha_{\g_1}(s) \alpha_{\g_2}'(0)
	\,(x\g_1)'\,\g_2 .
	\notag
\end{align}

Now fix a monomial $\g \in \BB$.  Assume $\alpha_{\g_1}$ has been
defined for all $\g_1 \fgt \g$.
Consideration of the coefficient of $x\g$
in (\ref{eq:pde}) gives us an equation
\begin{equation}\label{eq:sub}
	\alpha_\g'(s) = \alpha_\g'(0) + f_\g(s),
\end{equation}
where $f_\g(s)$ is a (real) linear combination
of terms $\alpha_{\g_1}(s) \alpha_{\g_2}'(0)$
where $\g_1, \g_2 \in \BB$
satisfy $\g \in \supp\big((x\g_1)'\g_2\big)$.  By
Lemma~\ref{lem:finite} (\,\fbox{W} Lemma~\ref{lemw:finite}), for a given
value of $\g$, there are
only finitely many pairs $\g_1, \g_2$ involved.

Now I claim these all have
$\g_1 \fgt \g$ and $\g_2 \fgt \g$.
Indeed, since $\g \in \supp\big((x\g_1)'\g_2\big)$ and
$(x\g_1)' = \g_1 + x\g_1'$, we have either $\g \fsteq \g_1\g_2$
or $\g \fsteq x\g_1'\g_2$.  Take two cases:
\begin{enumerate}
\item[(a)] $\g \fsteq \g_1\g_2$:  Now $\g_1 \fsteq \e$,
so $\g \fsteq \e\g_2 \fst \g_2$.  And $\g_2 \fsteq \e$,
so $\g \fsteq \e\g_1 \fst \g_1$.
\item[(b)] $\g \fsteq x\g_1'\g_2$:  Now $\g_1' \fst \g_1/(x\e)$
so $\g \fsteq x\g_1'\g_2 \fst \g_1\g_2/\e$.  But
$\g_1 \fsteq \e$, so $\g \fst \g_1$.  And
$\g_2 \fsteq \e$, so $\g \fst \g_2$.
\end{enumerate}
Thus we may use
equations (\ref{eq:sub})
to recursively define $\alpha_\g(s)$.  Indeed,
solving the differential equation, we get
$\alpha_\g(s) =  \int_0^s f_\g(u)\,du + s\alpha_\g'(0) + C$;
but $\alpha_\g(0) = 0$ and $\alpha_\g(1) = c_\g$, so
\begin{equation}\label{eq:subcrecursion}
	\alpha_\g(s) = s\left(c_\g - \int_0^1 f_\g(u)\,du\right)
	+ \int_0^s f_\g(u)\,du .
\end{equation}
In particular, the recursion begins with $\g = \e$ where
$f_\e(s) = 0$ and $\alpha_\e(s) = s c_\e = sa$.
Also (by induction) $\alpha_\g(s)$ and $f_\g(s)$ are
polynomials in $s$.
\end{proof}

\begin{thm}\label{thmcrit}
If (log-free) $T \sim x$ is moderate, then $T$ admits a real iteration
group with common support.
The coefficient functions are entire; we cannot conclude the
coefficient functions are polynomials.
\end{thm}
\begin{proof}
The proof is as in  Theorem~\ref{critical} above.   Here
are the details.  Write $T = x(1+U)$, $U \sim a\e$,
$\e \fst 1$.  Begin with the subgrid (\,\fbox{W} well ordered)
$\supp U$ which is contained in
$\SET{\g \in \G}{\g \fgteq \e, \g^\dagger \fsteq 1/(x\e)}$.
Let $\BB \supseteq \supp U$ be the least set such that
if $\g_1, \g_2 \in \BB$, then
$\supp\big((x\g_1)'\g_2\big) \subseteq \BB$.
By Lemma~\ref{lem:gensubgrid}, $\BB$ is a subgrid
(\,\fbox{W} by Lemma~\ref{lemw:gen}, $\BB$ is well ordered)
and $\BB \subseteq \SET{\g \in \G}{\g \fgteq \e,
\g^\dagger \fsteq 1/(x\e)}$.

Write $T(x) = x(1+\sum_{\g \in \BB} c_\g\g)$
and $\Phi(s,x) = x(1+\sum_{\g \in \BB} \alpha_\g(s)\g)$.
Compute the derivatives as in (\ref{eq:pdecomp}).

Now fix a monomial $\g \in \BB$.  Assume $\alpha_{\g_1}$
has been defined for all $\g_1 \fgt \g$.  There are two cases.
If $\g^\dagger \fst 1/(x\e)$, then the argument proceeds
as before, and we get (\ref{eq:subcrecursion}).
Now assume $\g^\dagger \fe 1/(x\e)$, say
$\g^\dagger \sim b/(x\e)$, $b \in \R$, $b \ne 0$.
Consideration of the coefficient of $x\g$
in (\ref{eq:pde}) gives us
an equation
\begin{equation}\label{eq:crit}
	\alpha_\g'(s) = \alpha_\g'(0) + f_\g(s) +
	b\alpha_\g(s) \alpha_\e'(0),
\end{equation}
where $f_\g(s)$ is a (real) linear combination
of terms $\alpha_{\g_1}(s) \alpha_{\g_2}'(0)$
where $\g_1, \g_2 \in \BB$
satisfy $\g \in \supp\big((x\g_1)'\g_2\big)$.
The term with $\g_1 = \g$, $\g_2 = \e$, namely
$b\alpha_\g(s) \alpha_\e'(0)$, is not included in $f_\g$
but written separately.
By Lemma~\ref{lem:finite} (\,\fbox{W} Lemma ~\ref{lemw:finite}), for a given
value of $\g$, there are
only finitely many pairs $\g_1, \g_2$ involved.

But I claim these all have
$\g_1 \fgt \g$ and $\g_2 \fgt \g$ except for
the case $\g_1=\g, \g_2=\e$.
Indeed, since $\g \in \supp\big((x\g_1)'\g_2\big)$ and
$(x\g_1)' = \g_1 + x\g_1'$, we have either $\g \fsteq \g_1\g_2$
or $\g \fsteq x\g_1'\g_2$. Take two cases:
\begin{enumerate}
\item[(a)]  $\g \fsteq \g_1\g_2$:  Now $\g_1 \fsteq \e$,
so $\g \fsteq \e\g_2 \fst \g_2$.  And $\g_2 \fsteq \e$,
so $\g \fsteq \e\g_1 \fst \g_1$.
\item[(b)] $\g \fsteq x\g_1'\g_2$:  Now $\g_1' \fsteq \g_1/(x\e)$
so $\g \fsteq x\g_1'\g_2 \fsteq \g_1\g_2/\e$, with $\fe$
only if $\g \fe x\g_1'\g_2$ and $\g_1^\dagger = 1/(x\e)$.  But
$\g_1 \fsteq \e$, so $\g \fsteq \g_1$, and $\fe$ would mean
$\g_1=\e$ and $\g_1^\dagger = 1/(x\e)$ so
$\e^\dagger \fe 1/(x\e)$, which is false; thus $\g \fst \g_1$.
And $\g_2 \fsteq \e$, so $\g \fsteq \g_2$.  This time $\fe$
means $\g_2 = \e$, $\g_1 = \g$, $\g^\dagger \fe 1/(x\e)$.
\end{enumerate}
Thus we may use
equations (\ref{eq:crit}) to recursively define $\alpha_\g(s)$.  Indeed,
solving the differential equation, we get
\begin{equation}\label{eq:crecursion}
	\alpha_\g(s) = \frac{e^{abs}-1}{e^{ab}-1}
	\left(c_\g - \int_0^1 e^{ab(1-u)}f_\g(u)\,du\right)
	+\int_0^s e^{ab(s-u)}f_\g(u)\,du
\end{equation}
if $\g^\dagger \sim b/(x\e)$.
\end{proof}

\begin{thm}\label{thm:deep}
If (log-free) $T \sim x$ is deep, then it does not admit a real iteration group
of the form {\rm(\ref{eq:phiseries})}.
\end{thm}

\begin{proof}
The proof is as in Proposition~\ref{super} above.
Here are the details.  Write $T = x(1+U)$, $U \sim a\e$,
$\e \fst 1$.  Suppose
$$
	\Phi(s,x) = x \left(1 + \sum_{\g\in \BB} \alpha_\g(s)\g\right)
$$
satisfies (\ref{eq:pde}), (\ref{eq:0}),
$\Phi(1,x) = T(x)$, and $\BB$ is a subgrid
(\,\fbox{W} well ordered).  We may assume
$\alpha_\g \ne 0$ for all $\g \in \BB$.
Write
$$
	\BB_1 = \SET{\g \in \BB}{\g^\dagger \fsteq 1/(x\e)},
	\qquad
	\BB_2 = \SET{\g \in \BB}{\g^\dagger \fgt 1/(x\e)}
$$
(the moderate and deep portions of $\BB$).
As before, $\max \BB = \e$ and $\alpha_\e(s) = as$,
$\alpha_\e'(0)=a$, $a \ne 0$.

Let $\m = \max\BB_2$.  So $\m^\dagger \fgt 1/(x\e)$.
Write $\m^\dagger \sim -b \p / (x\e)$, $b \in \R$, $b > 0$,
$\p \in \G$, $\p \fgt 1$.  ($\m$ is small and positive, so
$\m'$ is negative.)
In (\ref{eq:pde}) we will consider the coefficient of
$x\p\m = x\mag\big((x\m)'\e\big)$.
By Lemma~\ref{lem:logder}(g),
$((x\m)')^\dagger \sim \m^\dagger \fgt 1/(x\e)$;
by Remark~\ref{ecase},
$\e^\dagger \fst 1/(x\e)$; so
$(\p\m)^\dagger = ((x\m)')^\dagger + \e^\dagger
\sim ((x\m)')^\dagger \fgt 1/(x\e)$
and $\p\m \notin\BB_1$.
Also $\p\m \fgt \m$, so $\p\m \notin\BB_2$.  Thus
$\p\m \notin \BB$ so $\alpha_{\p\m} = 0$.

I claim the only $\g_1,\g_2 \in \BB$ with
$\p\m \in \supp\big((x\g_1)'\g_2\big)$ are $\g_1=\m$, $\g_2 = \e$.
To prove this, consider these cases:
\begin{enumerate}
\item[(a)] $\g_1 \fgt \m, \g_2 \fgt \m$.  By Lemma~\ref{lem:group}(b),
$\supp((x\g_1)') \subseteq \BB_1$ and by Lemma~\ref{lem:group}(a),
$\supp((x\g_1)'\g_2) \subseteq \BB_1$.  So
$\p\m \notin \supp((x\g_1)'\g_2)$.
\item[(b)] $\g_1 = \m, \g_2=\e$. In this case
$\p\m = \mag((x\g_1)'\g_2)$, so $\p\m \in \supp((x\g_1)'\g_2)$. 
\item[(c)] $\g_1 = \m, \g_2 \fst \e$.  Then
$(x\g_1)'\g_2 \fst (x\m)'\e \fe \p\m$,
so $\p\m \notin \supp((x\g_1)'\g_2)$.\item[(d)] $\g_1 \fst \m$.  Then $(x\g_1)'\g_2 \fst (x\m)'\e \fe \p\m$,
so $\p\m \notin \supp((x\g_1)'\g_2)$.
\item[(e)] $\g_2 \fsteq \m$.  Then by Remark~\ref{ecase},
$(x\g_1)'\g_2 \fsteq (x\e)'\m \fst \m \fst \p\m$,
so $\p\m \notin \supp((x\g_1)'\g_2)$.
\end{enumerate}
So consideration of the coefficient of $x\p\m$ in (\ref{eq:pde}) yields:
$$
	0 = -b\alpha_\m(s)\alpha_\e'(0) = -ab\alpha_\m(s),
$$
where $-ab \ne 0$, so $\alpha_\m = 0$, a contradiction.
\end{proof}

\begin{re}
If $T \sim x$, then a calculation shows that $T$ is
shallow, moderate, or deep if and only if
$\log\circ\, T \circ\exp$ is shallow, moderate, or deep, respectively.
This invariance of the classification will show that the three
theorems above are correct for all $T \in \T$ with $T \sim x$,
even if not log-free.  For example,
$T = x+\log x$ is shallow, since log-free
$$
	\log\circ\;T\circ\exp = x + \sum_{j=1}^\infty
	\frac{(-1)^{j+1}}{j}\,x^j e^{-jx}
$$
is shallow.

Recall that if $T$ is large and positive and has exponentiality $0$,
then for some $k$ we have $\log^{[k]} \circ\, T \circ \exp^{[k]} \sim x$.
We might try to use the same principle
to extend the definitions of shallow, moderate,
and deep to such transseries $T$ even if $T \not\sim x$.
Define $T$ \Def{shallow} provided $\log^{[k]} \circ\, T \circ \exp^{[k]}$
is shallow for some $k$; similarly for moderate and deep.
Examples:
$T = x\log x$ is shallow, since $\log\circ\,T\circ\exp = x + \log x$
is shallow.  The finite power series $U = 2x-2/x$ is moderate, since
$$
	\log\circ\, U\circ\exp = x + \log 2
	- \sum_{j=1}^\infty \frac{e^{-2jx}}{j}
$$
is moderate.
And $V = 2x - 2e^{-x}$ is deep.

But the usefulness of this extension is not entirely clear,
since it may produce a family $\Phi(s,x)$ without
common support.  Example:
For $T = x\log x$ we compute $S = \log\circ\,T\circ\exp = x + \log x$.
So we get a real iteration group for $S$ of the form
$\Psi(s,x) = x + s\log x + \o(1)$, and then a
real iteration group for $T$ of the form
$\Phi(s,x) = \exp(\Psi(s,\log x))= x(\log x)^s +\cdots$.
These are not all supported by a common subgrid or even
a common well ordered set.  When the support
depends on the parameter $s$, it may no longer make
sense to require the coefficients be differentiable.

Example $x^2+c$ is also deep.  It is discussed in Section~\ref{sec:julia},
below.  The figure there illustrates supports of iterates $M^{[s]}$ that
vary with $s$.
\end{re}

\subsection*{Technical Lemmas}

\begin{lem}\label{lem:logder}
Let $\fa, \fb \in \G$ and $A, B \in \T$.
{\rm(a)}~If $\fa = e^A$, with $A \ne 0$ purely large, then
$\fa^\dagger = A'$.
{\rm(b)}~$(AB)^\dagger = A^\dagger + B^\dagger$.
{\rm(c)}~If $A = a \g (1+U)$, $\g\ne 1$, $U \fst 1$,
then $A^\dagger \sim \g^\dagger$.
{\rm(d)}~If $1 \fst \fa \fsteq \fb$, then
$\fa^\dagger \fsteq \fb^\dagger$.
{\rm(d$'$)}~If $1 \fst A \fsteq B$, then $A^\dagger \fsteq B^\dagger$.
{\rm(e)}~If $1 \fgt \fa \fgteq \fb$, then
$\fa^\dagger \fsteq \fb^\dagger$.
{\rm(e$'$)}~If $1 \fgt A \fgteq B$, then $A^\dagger \fsteq B^\dagger$.
{\rm(f)}~If $\fb \fst \fa \fst 1/\fb$, then
$\fa^\dagger \fsteq \fb^\dagger$.
{\rm(g)}~If $\fb \ne 1$ is log-free and $\n \in \supp\big((x\fb)'\big)$, then $\n^\dagger \sim \fb^\dagger$.
\end{lem}
\begin{proof}
(a) $\fa' = A'e^A$ so $\fa^\dagger = \fa'/\fa = A'$.

(b) Product rule.

(c) Since $\g = e^L$, $L$ purely large, note
$\g^\dagger = L'$.  Also $U \fst 1$ so
$(1+U)^\dagger = U'/(1+U) \sim U'$.  Now
$L \fgt 1 \fgt U$, so $L' \fgt U'$ and
$\g^\dagger \fgt (1+U)^\dagger$.  So
$A^\dagger = a^\dagger + \g^\dagger + (1+U)^\dagger \sim \g^\dagger$.

(d) Write $\fa = e^A, \fb = e^B$, with $A,B$ purely large.
Now $0 < A \le B$, so $A \fsteq B$.  Also $B \not\fe 1$, so
$A' \fsteq B'$, that is, $\fa^\dagger \fsteq \fb^\dagger$.
(d$'$) follows from (c).

(e) Write $\fa = e^A, \fb = e^B$, with $A,B$ purely large.
Now $0 > A \ge B$, so $A \fsteq B$.  Also $B \not\fe 1$, so
$A' \fsteq B'$, that is, $\fa^\dagger \fsteq \fb^\dagger$.
(e$'$) follows from (c).

(f) If $\fa=1$, then $\fa^\dagger = 0$.
If $\fa \fst 1$, then apply (e).  If $\fa \fgt 1$,
note that $\fb^\dagger = -(1/\fb)^\dagger \fe (1/\fb)^\dagger$,
then apply (d).

(g) Case 1: $\fb = x^b$, $b \in \R$, $b \ne 0$.
Then $x\fb = x^{b+1}$, $(x\fb)' = (b+1) x^b \sim \fb$.  So
$((x\fb)')^\dagger \sim \fb^\dagger$.
Case 2: $\fb = x^b e^L$, $L \ne 0$ has height $n \ge 0$.
Assume $\fb \fst 1$, the case $\fb \fgt 1$ is similar.
Then $x\fb = x^{b+1}e^L$, $(x\fb)' = ((b+1)+xL')x^be^L = A\fb$
where $A = (b+1)+xL'$ has height $n$.  Now
$\n \in \supp\big((x\fb)'\big)$, so $\n = \fa\fb$ where
$\fa \in \supp A$.  But $\fa$ has height $n$, so by ``height wins''
\cite[\Bheightwins]{edgar}, we have
$\fb \fst \fa \fst 1/\fb$.  In the proofs of
(d) and (e), if it is a case of ``height wins'' then
we get strict inequality
$A \fst B$ and $\fa^\dagger \fst \fb^\dagger$.
So here $\fa^\dagger \fst \fb^\dagger$.
Therefore $\n^\dagger = (\fa\fb)^\dagger
= \fa^\dagger + \fb^\dagger \sim \fb^\dagger$.
\end{proof}

\begin{lem}\label{lem:eT}
Let $T = x(1+a\e+ \o(\e))$ with $\e \in \G$, $\e \fst 1$,
$a \in \R$, $a > 0$.  Let $\g \in \G$, $\g \fst 1$.
Then
{\rm(i)}~$\e(T) \sim \e$;
{\rm(ii)}~if $\g \fst 1/(x\e)$ then $\g(T) \sim \g$;
{\rm(iii)}~if $\g \sim -b/(x\e)$, $b \in \R$ then $\g(T) \sim e^{-ab}\g$;
{\rm(iv)}~if $\g \fgt 1/(x\e)$, then $\g(T) \fst \g$.
\end{lem}
\begin{proof}
(i) As in Remark~\ref{ecase}, $\e' \fst 1/x$.  Then
$$
	\e(T) - \e = \int_x^T \e' \fst \int_x^T \frac{1}{x}
	= \log\frac{T}{x} = \log(1 + a\e + \o(\e)) \sim a\e \fe \e,
$$
so $\e(T) \sim \e$.

(ii) Assume $\g^\dagger \fst 1/(x\e)$.  Then
$$
	\log\frac{\g(T)}{\g} = \int_x^T \g^\dagger \fst
	\int_x^T \frac{1}{x\e}.
$$
By \cite[\Cmvtii]{edgarc}, the value of this integral is between
$$
	\frac{T-x}{x\e} \sim \frac{ax\e}{x\e} = a
	\qquad{and}\qquad
	\frac{T-x}{T\e(T)} \sim \frac{ax\e}{x\e} = a .
$$
Thus $\log(\g(T)/\g) \fst 1$ so $\g(T)/\g \sim 1$.

(iii) Assume $\g^\dagger \sim -b/(x\e)$.  Then
$$
	\log\frac{\g(T)}{\g} = \int_x^T \g^\dagger
	\sim \int_x^T \frac{-b}{x\e} \sim -ab,
$$
where the integral was estimated in the same way as in (ii).
Therefore $\g(T) \sim e^{-ab}\g$.

(iv) Assume $\g^\dagger \fgt 1/(x\e)$.
Then
$$
	\log\frac{\g(T)}{\g(x)} =
	\int_x^T \g^\dagger \fgt
	\int_x^T \frac{1}{x\e} \sim a .
$$
But $\log(\g(T)/\g(x))<0$ so $\g(T)/\g(x) \fst 1$
and $\g(T) \fst \g$.
\end{proof}

For completeness, we note the following analogous version
for $T < x$.

\begin{lem}\label{lem:eTx}
Let $T = x(1-a\e+ \o(\e))$ with $\e \in \G$, $\e \fst 1$,
$a \in \R$, $a > 0$.  Let $\g \in \G$, $\g \fst 1$.
Then
{\rm(i)}~$\e(T) \sim \e$;
{\rm(ii)}~if $\g \fst 1/(x\e)$ then $\g(T) \sim \g$;
{\rm(iii)}~if $\g \sim -b/(x\e)$, $b \in \R$ then $\g(T) \sim e^{ab}\g$;
{\rm(iv)}~if $\g \fgt 1/(x\e)$, then $\g(T) \fgt \g$.
\end{lem}

\begin{lem}\label{lem:group}
Let $\m \in \G$ be a monomial.  Let
$\BB = \SET{\g \in \G}{\g^\dagger \fst \m}$ and
$\widetilde\BB = \SET{\g \in \G}{\g^\dagger \fsteq \m}$.
Then: {\rm(a)}~$\BB$ and $\widetilde\BB$ are subgroups of $\G$.
{\rm(b)}~Let $\g \in \G$ be log-free and small.
If $\g \in \BB$, then $\supp\big((x\g)'\big) \subseteq \BB$.
If $\g \in \widetilde\BB$, then
$\supp\big((x\g)'\big) \subseteq \widetilde\BB$.
{\rm(c)}~Let $\g = x^be^L$, where $L$ is purely large
and log-free.  If $\g \in \BB$, then
$\supp L \subseteq \BB$.
If $\g \in \widetilde\BB$, then
$\supp L \subseteq \widetilde\BB$.
\end{lem}
\begin{proof}
(a)  $(1/\g)^\dagger = -\g^\dagger$, so if $\g \in \BB$, then
$1/\g \in \BB$ and if $\g \in \widetilde\BB$, then
$1/\g \in \widetilde\BB$.  And
$(\g_1 \g_2)^\dagger = \g_1^\dagger + \g_2^\dagger$,  so
if $\g_1, \g_2 \in \BB$, then $\g_1 \g_2 \in \BB$, and
if $\g_1, \g_2 \in \widetilde\BB$, then $\g_1 \g_2 \in \widetilde\BB$.

(b)  Apply Lemma~\ref{lem:logder}(g).

(c) Let $\n \in \supp L$.  Then by ``height wins''
we have $n^\dagger \fsteq \g^\dagger$.
\end{proof}

\begin{lem}\label{lem:deriv}
{\rm\fbox{G}} Let $\AA$ be a subgrid.  Then
$\bigcup_{\fa \in \AA} \supp(\fa')$
is also a subgrid.  For any given $\g \in \G$, there
are only finitely many $\fa \in \AA$ with $\g \in \supp(\fa')$.
\end{lem}
\begin{proof} \cite[\Wgridderiv]{edgarw}.
(This is exacly what is needed for
the proof that the derivative $(\sum_{\fa\in\AA}\fa)'$
is defined.)
\end{proof}

\begin{lem}\label{lemw:deriv}
{\rm\fbox{W}} Let $\AA$ be well ordered.  Then
$\bigcup_{\fa \in \AA} \supp(\fa')$
is also well ordered.  For any given $\g \in \G$, there
are only finitely many $\fa \in \AA$ with $\g \in \supp(\fa')$.
\end{lem}
\begin{proof} \cite[Lem.~3.2]{DMM} or \cite[\Cwellderiv]{edgarc}.
\end{proof}

\begin{lem}\label{lem:prod}
{\rm\fbox{G}} Let $\AA$ and $\BB$ be subgrids.  Then
$\AA\BB$ is a subgrid.  If $\g \in \AA\BB$, then there
are only finitely many pairs $\fa \in \AA, \fb \in \BB$
with $\fa\fb=\g$.
\end{lem}
\begin{proof}
\cite[\Bgridprod]{edgar} and \cite[\Bwellproduct]{edgar}.
\end{proof}

\begin{lem}\label{lemw:prod}
{\rm\fbox{W}} Let $\AA$ and $\BB$ be well ordered.  Then
$\AA\BB$ is well ordered.  If $\g \in \AA\BB$, then there
are only finitely many pairs $\fa \in \AA, \fb \in \BB$
with $\fa\fb=\g$.
\end{lem}
\begin{proof}
\cite[\Bwellproduct]{edgar}.
\end{proof}

\begin{lem}\label{lem:finite}
{\rm\fbox{G}} Let $\BB$ be a subgrid.  Let $\g \in \G$.
There are only finitely many pairs $\g_1,\g_2 \in \BB$
such that $\g \in \supp\big((x\g_1)'\g_2\big)$.
\end{lem}
\begin{proof}
Since $\BB$ is a subgrid, also
$\BB_1 := \SET{x\g}{\g \in \BB}$ is a subgrid.
By Lemma~\ref{lem:deriv},
$\BB_2 := \bigcup_{\g \in \BB} \supp((x\g)')$ is a subgrid.
Then $\BB_2 \BB = \bigcup_{\g_1,\g_2 \in \BB} \supp((x\g_1)'\g_2)$,
and by Lemma~\ref{lem:prod}, for any given $\g \in \G$,
there are only finitely many
pairs $\g_1,\g_2 \in \BB$
such that $\g \in \supp\big((x\g_1)'\g_2\big)$.
\end{proof}

\begin{lem}\label{lemw:finite}
{\rm\fbox{W}} Let $\BB$ be well ordered.  Let $\g \in \G$.
There are only finitely many pairs $\g_1,\g_2 \in \BB$
such that $\g \in \supp\big((x\g_1)'\g_2\big)$.
\end{lem}
\begin{proof}
Since $\BB$ is well ordered, also $\BB_1 := \SET{x\g}{\g \in \BB}$
is well ordered.  By Lemma~\ref{lemw:deriv},
$\BB_2 := \bigcup_{\g \in \BB} \supp((x\g)')$ is well ordered,
and each monomial in $\BB_2$ belongs to $\supp((x\g)')$ for
only finitely many $\g$.  Then
$\BB_2\BB = \bigcup_{\g_1,\g_2\in\BB} \supp((x\g_1)'\g_2)$,
and by Lemma~\ref{lemw:prod}, for any given $\g \in \G$,
there are only finitely many pairs $\g_2,\g_2 \in \BB$
such that $\g \in \supp((x\g);\g_2)$.
\end{proof}

\begin{lem}\label{lem:closure}
Let $\e \in \G$, $\e \fst 1$.  Let
$$
	\AA = \SET{\g \in \G}{\g \fsteq \e, \g^\dagger \fst \frac{1}{x\e}} ,
	\qquad
	\widetilde\AA = \SET{\g \in \G}{\g \fsteq \e,
	\g^\dagger \fsteq \frac{1}{x\e}} .
$$
{\rm(a)}  If $\g_1, \g_2 \in \AA$, then $\g_1 \g_2 \in \AA$.
If $\g_1, \g_2 \in \widetilde\AA$, then $\g_1 \g_2 \in \widetilde\AA$.
{\rm(b)}  If $\g_1, \g_2 \in \AA$, then
$\supp\big((x\g_1)'\g_2\big) \subseteq \AA$.
If $\g_1, \g_2 \in \widetilde\AA$, then
$\supp\big((x\g_1)'\g_2\big) \subseteq \widetilde\AA$.
{\rm(c)}  If $\g \in \AA$, then
$\supp(x\e\g') \subseteq \AA$.
If $\g \in \widetilde\AA$, then
$\supp(x\e\g') \subseteq \widetilde\AA$.
\end{lem}
\begin{proof}
(a) If $\g_1, \g_2 \fsteq \e$, then
$\g_1 \g_2 \fsteq \e \e \fst \e$.  Combine this with
Lemma~\ref{lem:group}(a).

(b) If $\g_1 \fsteq \e$, then $x\g_1 \fsteq x\e$ and
$(x\g_1)' \fsteq (x\e)'$.  Now $\e \fst 1$ so
$x\e \fst x$ and $(x\e)' \fst 1$.  If
$\g_2 \fsteq \e$ also, then
$(x\g_1)'\g_2 \fsteq (x\e)'\e \fst 1\e = \e$.
Combine this with Lemma~\ref{lem:group}(b).

(c) is similar, noting that $\e^\dagger \fst 1/(x\e)$
and $x^\dagger \fst 1/(x\e)$.
\end{proof}

\begin{lem}\label{lem:gensubgrid}
{\rm\fbox{G}}
Let $\BB \subset \Gsmall$ be a log-free subgrid.  Write
$\e = \max\BB$ and assume
$\BB \subseteq \SET{\g \in \G}{\g \fsteq \e,\g^\dagger \fsteq 1/(x\e)}$.
Let $\widetilde\BB$ be the least subset of $\G$ such that

{\rm (i)} $\widetilde\BB \supseteq \BB$,

{\rm(ii)} if $\g_1, \g_2 \in \widetilde\BB$, then
$\g_1\g_2 \in \widetilde\BB$,

{\rm(iii)} if $\g_1, \g_2 \in \widetilde\BB$, then
$\supp\big((x\g_1)'\g_2\big) \subseteq \widetilde\BB$.

\noindent Then $\widetilde\BB$ is a subgrid.
\end{lem}
\begin{proof}
Let $\AA$ be the least subset of $\G$ such that

(i) $\AA \supseteq \BB$,

(ii) if $x^b e^L \in \AA$, then $\supp L \subseteq \AA$.

\noindent By \cite[\Wheredsubgrid]{edgar}, $\AA$ is a subgrid.
From Lemma~\ref{lem:group}(c) we have $\g^\dagger \fsteq 1/(x\e)$
for all $\g \in \AA$.  
There is a ratio set $\bmu = \{\mu_1,\cdots,\mu_n\}$,
chosen from the group generated
by $\AA$, so that $\AA \subseteq \GRID^\ebmu$.  Because they come
from the group generated by $\AA$, we have $\mu_i^\dagger \fsteq 1/(x\e)$
for $1 \le i \le n$ by Lemma~\ref{lem:group}(a).  Remark that
$x^\dagger \fst 1/(x\e)$ since $\e \fst 1$.  And
$\e^\dagger \fst 1/(x\e)$ was noted in Remark~\ref{ecase}.
So we may without harm add more generators to $\bmu$
and assume $\e,x \in \GRID^\ebmu$.
This has been arranged so that
if $\g \in \GRID^\ebmu$, then $\supp(\g') \subseteq \GRID^\ebmu$.
Now $x\e\mu_i^\dagger \fsteq 1$, so
$$
	\{\e, x^{-1}\} \cup \bigcup_{i=1}^n \supp(x\e\mu_i^\dagger)
	\cup \frac{1}{\e}\BB
$$
is a finite union of subgrids, so it is itself a subgrid.
All of its elements are ${}\fsteq 1$, so
\cite[\Baddendumsmall]{edgar} there is a ratio set $\ba$ such that
$\GRID^{\ba,\0}$ contains that finite union.  Again all
elements of $\ba$ may be chosen from the group
generated by $\GRID^\ebmu$.  
So all $\fa \in \GRID^{\ba,\0}$
still satisfy $\fa^\dagger \fsteq 1/(x\e)$.

To complete the proof that $\widetilde\BB$ is a subgrid,
I will show that $\widetilde\BB \subseteq \e\GRID^{\ba,\0}$.
First, note that $\BB \subseteq \e\GRID^{\ba,\0}$.
Next, if $\g_1, \g_2 \in \e\GRID^{\ba,\0}$, then
$\g_1\g_2 \in \e(\GRID^{\ba,\0}\e\GRID^{\ba,\0})
\subseteq \e\GRID^{\ba,\0}$.  Finally, suppose
$\g_1,\g_2 \in \e\GRID^{\ba,\0}$.  Because $\ba$ is from
the group $\GRID^\ebmu$, we may write
$\g_1 = \mu_1^{k_1}\cdots\mu_n^{k_n}$, and
\begin{align*}
	\g_1^\dagger &= k_1\mu_1^\dagger +\cdots+
	k_n\mu_n^\dagger ,
	\\
	x\e\g_1^\dagger &= k_1x\e\mu_1^\dagger +\cdots+
	k_nx\e\mu_n^\dagger ,
\end{align*}
so that $\supp (x\e\g_1^\dagger) \subseteq \GRID^{\ba,\0}$.
Also $\g_1/\e \in \GRID^{\ba,\0}$ and $\g_2/\e \in \GRID^{\ba,\0}$.
Therefore
\begin{equation*}
	\supp\big(x\g_1'\g_2\big) = \e\;
	\left(\frac{\g_1}{\e}\right)\left(\frac{\g_2}{\e}\right)
	\supp\left(x\e\g_1^\dagger\right)
	\subseteq \e\GRID^{\ba,\0} .
\end{equation*}
And $(x\g_1)'\g_2 = \g_1\g_2 + x\g_1'\g_2$, so
$\supp\big((x\g_1)'\g_2\big) \subseteq \e\GRID^{\ba,\0}$.
By the definition of $\widetilde\BB$ we have
$\widetilde\BB \subseteq \e\GRID^{\ba,\0}$, and it
is therefore a subgrid.
\end{proof}

\begin{lem}\label{lemw:gen}
{\rm\fbox{W}}
Let $\BB \subset \Gsmall$ be log-free and well ordered.  Write
$\e = \max\BB$ and assume
$\BB \subseteq \SET{\g \in \G}{\g \fsteq \e,\g^\dagger \fsteq 1/(x\e)}$.
Let $\widetilde\BB$ be the least subset of $\G$ such that

{\rm (i)} $\widetilde\BB \supseteq \BB$,

{\rm(ii)} if $\g_1, \g_2 \in \widetilde\BB$, then
$\g_1\g_2 \in \widetilde\BB$,

{\rm(iii)} if $\g_1, \g_2 \in \widetilde\BB$, then
$\supp\big((x\g_1)'\g_2\big) \subseteq \widetilde\BB$.

\noindent Then $\widetilde\BB$ is well ordered.
\end{lem}
\begin{proof}
Let $\BB_1$ be the least set such that
$\BB_1 \supseteq \BB \cup\{\e^2\}$ and
if $\g \in \BB_1$, then $\supp(x\e\g') \subseteq \BB_1$.
Then $\BB_1$ is well ordered by \cite[\Cuii]{edgarc}.
For all $\g \in \BB_1$ we have $\g\fsteq\e$
and $\g^\dagger \fsteq1/(x\e)$ by Lemma~\ref{lem:closure}(c).
Still $\e=\max\BB_1$, $\e^2 \in \BB_1$, and $\supp(x\e\e')\subseteq \BB_1$.

Let $\BB_2 = \e^{-1}\BB_1$.  Then $\BB_2$
is well ordered, $\BB_2 \supseteq \e^{-1}\BB$,
$1 = \max\BB_2$, $\e \in \BB_2$,
$\supp(x\e')\subseteq \BB_2$.
If $\m \in \BB_2$, then $\supp\big(x(\e\m)'\big)\subseteq \BB_2$.

Let $\BB_3$ be the semigroup generated by $\BB_2$.
Then $\BB_3$ is well ordered, $\BB_3 \supseteq \e^{-1}\BB$,
$1 = \max\BB_3$, $\e \in \BB_3$, $\supp(x\e')\subseteq \BB_3$.
From the identity
$$
	x(\e\m_1\m_2)' = x(\e\m_1)'\cdot\m_2 +
	\m_1\cdot x(\e\m_2)' - x\e'\m_1\m_2
$$
we conclude:
if $\m \in \BB_3$, then $\supp\big(x(\e\m)'\big) \subseteq \BB_3$.

Finally, let $\BB_4 = \e\BB_3$.  Then
$\BB_4$ is well ordered,
$\BB_4 \supseteq \BB$, $\e = \max\BB_4$.
Let $\g_1, \g_2 \in \BB_4$.  Then $\g_1/\e, \g_2/\e \in \BB_3$,
so $(\g_1/\e)\cdot(\g_2/\e) \in \BB_3$ and
$\g_1\g_2/\e^2 \in \BB_3$.  Now $\e \in \BB_3$ so
$\g_1\g_2/\e \in \BB_3$ and therefore $\g_1\g_2 \in \BB_4$.
Again let $\g_1,\g_2 \in \BB_4$.  Then
$\g_1/\e, \g_2/\e \in \BB_3$.  So $\supp(x\g_1')\subseteq \BB_3$.
Thus $\supp(x\g_1')\g_2/\e \subseteq \BB_3$ so
$\supp(x\g_1')\g_2 \subseteq \BB_4$.
And $(x\g_1)'\g_2 = \g_1\g_2 + x\g_1'\g_2$,
so we conclude $\supp\big((x\g_1)'\g_2\big) \subseteq \BB_4$.

This shows $\BB_4 \supseteq \widetilde\BB$, and therefore
that $\widetilde\BB$ is well ordered.
\end{proof}

\section{Abel's Equation}
Let $T \in \LP$, $T > x$. \Def{Abel's Equation}
for $T$ is $V(T(x)) = V(x)+1$.  If large positive
$V$ exists satisfying this, then a real iteration group
$\Phi$
may be obtained as $\Phi(s,x) = V^{[-1]}\circ(x+s)\circ V$.
(In general such $\Phi$ will not have common support.)
If $T \in \LP$, $T < x$, then Abel's Equation for $T$ is
$V(T(x)) = V(x)-1$, and then we may similarly write
$\Phi(s,x) = V^{[-1]}\circ(x-s)\circ V$.

We now do this in reverse: Let $\Phi(s,x)$ be of the
form constructed as in Theorem~\ref{thmcrit}, that is,
coefficients defined recursively by
(\ref{eq:subcrecursion}) and~(\ref{eq:crecursion}).
Then we can
use $\Phi$ to get $V$ for Abel's Equation.

\begin{thm}\label{mabel}
{\rm(Moderate Abel)}
Let $T \sim x$, $T > x$ be moderate, and let
$\Phi(s,x) \in \LP$ for all $s \in \R$ be the real iteration
group for $T$ constructed in Theorem~\ref{thmcrit}.
Then
$$
	V(x) := \int \frac{dx}{\Phi_1(0,x)}
$$
is large and positive and satisfies Abel's Equation $V(T(x)) = V(x)+1$.
\end{thm}
\begin{proof}
Now $V' \sim 1/(ax\e) \fgt 1/x$, so $V \fgt \log x$ is large.
And $V' > 0$ so $V > 0$.  (A large negative transseries has
negative derivative.)

From $\Phi(s+t,x) = \Phi(s,\Phi(t,x))$
take $\partial/\partial s$ then
substitute $t=1, s=0$ to get $\Phi_1(1,x) = \Phi_1(0,T)$.
As constructed, $\Phi_1(s,x) = \Phi_2(s,x)\Phi_1(0,x)$.
So $\Phi_2(1,x)\Phi_1(0,x) = \Phi_1(1,x) = \Phi_1(0,T)$.
Now from $\Phi(1,x) = T$ we have $\Phi_2(1,x) = T'$.
So
$$
	\frac{T'}{\Phi_1(0,T)} = \frac{1}{\Phi_1(0,x)},
$$
or $V'(T)\cdot T' = V'(x)$ so $V(T) = V+c$ for some $c \in \R$.
Now
$$
	V(T) - V(x) = \int_x^T V' = \int_x^T \frac{1}{\Phi_1(0,x)}
	\sim \int_x^T \frac{1}{ax\e}.
$$
By \cite[\Cmvtii]{edgarc}, this integral is between
$$
	\frac{T-x}{ax\e} \sim \frac{ax\e}{ax\e} = 1
	\qquad\text{and}\qquad
	\frac{T-x}{aT\e(T)} \sim \frac{ax\e}{ax\e} = 1.
$$
We used $\e(T) \sim \e$ from Lemma~\ref{lem:eT}(i).
So we have $V(T)-V \sim 1$ and thus $c=1$.
\end{proof}

Now we will consider the deep case.  For $T = x+1 + A$,
consider Abel's Equation $V\circ T = V+1$.  A formal
solution is
\begin{equation}\label{eq:formalabel}
	V = x + A + A\circ T + A\circ T^{[2]} + A\circ T^{[3]} + \cdots .
\end{equation}
But if $T$ is not purely deep, then $A \circ T \fe A$
(Lemma~\ref{lem:eTx}), so series (\ref{eq:formalabel})
does not converge.  We will use the moderate version already proved
(Theorem~\ref{mabel})
to reduce the general case to one where
$A^\dagger \fgt 1$ (Proposition~\ref{reduce})
so that $A\circ T \fst A$
and the series does converge (Proposition~\ref{pdabel}).
But in general it cannot be grid-based
(Example~\ref{not_grid}), so the
final step works only for the well-based version of $\T$.

\begin{pr}\label{reduce}
Let $T \sim x$, $T > x$.  There exists large positive $V$ such that
$V \circ T \circ V^{[-1]} = x + 1 + B$ and
$\g^\dagger \fgt 1$ for all $\g \in \supp B$; that is,
$x+1+B$ is purely deep.
Let $0 < S < x$, $S \fgt 1$.  There exists $V \in \LP$ such that
$V \circ S \circ V^{[-1]} = x - 1 + C$ and
$\g^\dagger \fgt 1$ for all $\g \in \supp C$.
\end{pr}
\begin{proof}
Write $T = x(1 + a\e + A_1 + A_2)$, where

(i) $\g \fst \e, \g^\dagger \fsteq 1/(x\e)$ for all $\g \in \supp A_1$,

(ii) $\g \fst 1, \g^\dagger \fgt 1/(x\e)$ for all $\g \in \supp A_2$.

\noindent
So $T_1 = x(1+a\e+A_1)$ is the moderate part of $T$ (including the shallow
part), and $T - T_1 = xA_2$ is the deep part of $T$. 
By Theorem~\ref{mabel}, there
is large positive $V \fgt \log x$ so that
$V \circ T_1 = V+1$ and $V' \sim 1/(ax\e)$.
So compute
$$
	V \circ T - V \circ T_1 =
	\int_{T_1}^T V' \fgt \int_{T_1}^T\frac{1}{x} .
$$
Now $T-T_1 = A_2$, so this integral is between
$A_2/T_1 \sim A_2/x$ and $A_2/T \sim A_2/x$.
So $B_1 := V \circ T - V \circ T_1 \fst A_2$
and
$$
	V\circ T = V\circ T_1 + B_1 = V + 1 + B_1 
$$
with $B_1^\dagger \fgt 1/(x\e)$.  So write
$B = B_1\circ V^{[-1]}$ to get
$V \circ T \circ V^{[-1]} = x + 1 + B$ and
$B_1^\dagger = (B\circ V)^\dagger = (B^\dagger \circ V)\cdot V'$.
But $V' \fe 1/(x\e)$ and $B_1^\dagger \fgt 1/(x\e)$, so
$B^\dagger \fgt 1$.

Now let $0<S<x$, $S \fgt 1$.  Then define $T := S^{[-1]}$ to get
$T>x$, $T \fgt 1$.  So as we have just seen, there is
$V$ with $V\circ T \circ V^{[-1]} = x + 1 + B$.
Take the inverse to get
\begin{align*}
	V\circ S \circ V^{[-1]} = \big(V\circ T \circ V^{[-1]}\big)^{[-1]}
	= (x + 1 + B)^{[-1]} .
\end{align*}
So if $(x + 1 + B)^{[-1]} = (x-1+C)$, we must show $C^\dagger \fgt 1$.
Now $(x+1+B)\circ(x-1+C) = x$, so
$x-1+C+1+B\circ(x-1+C)=x$ and therefore
$C = -B\circ(x-1+C)$, so
$$
	C^\dagger = (B^\dagger\circ(x-1+C))\cdot(x-1+C)'
	\fgt 1\cdot 1 = 1,
$$
as required.
\end{proof}

The following proof is only for the well-based version of $\T$.
In Example~\ref{not_grid}, below, we see it fails in general
for the grid-based version of $\T$.

\begin{pr}\label{pdabel}
{\rm \fbox{W} (Purely Deep Abel)}
{\rm(a)}~Let $T = x+1+A$, $A \in \T$, $A \fst 1$, $A^\dagger \fgt 1$.
There is $V = x+B$, $B \in \T$,  $B \fst 1$,
$B^\dagger \fgt 1$, such that
$V\circ T = V+1$.
{\rm(b)}~Let $T = x-1+A$, $A \in \T$, $A \fst 1$, $A^\dagger \fgt 1$.
There is $V = x+B$, $B \in \T$,  $B \fst 1$,
$B^\dagger \fgt 1$ such that
$V\circ T = V-1$.
\end{pr}
\begin{proof}
(a) There exist $N,M \in \N$ so that
$\supp A \subset \G_{N,M}$.  Increase $N$ if necessary so that
$N \ge M$ and $x \in \G_{N,M}$.
Now $\e = x^{-1}$, so the deep monomials are
$\DD = \SET{\g \in \Gsmall_{N,M}}{\g^\dagger \fgt 1}$.
I claim: if $B \in \T$, $\supp B \subseteq \DD$,
then $B(T) \fst B$ and
$\supp B(T) \subseteq \DD$.  Indeed, all $\g \in \supp B$
satisfy $\g(T) \fst \g$ by Lemma~\ref{lem:eT}(iv), and we may sum
to conclude $B(T) \fst B$.  Therefore
$B(T)^\dagger \fgteq B^\dagger \fgt 1$.
And $\supp B(T) \subseteq \G_{N,M}$ by \cite[\BcompoNM]{edgar}.

Let $\SA = \SET{x+B \in \T}{\supp B \subseteq \DD}$.
Define $\Psi$ by $\Psi(Y) := Y\circ T - 1$.
We want to apply a fixed point argument to $\Psi$.
First we must show that $\Psi$ maps $\SA$ into $\SA$.
Let $x+B \in \SA$.  So $\Psi(x+B) = T + B\circ T - 1
= x + 1 + A + B(T) - 1$.  But $\supp A \subseteq \DD$,
$\supp B \subseteq \DD$ so $\supp B(T) \subseteq \DD$,
and thus $x + A+B(T) \in \SA$.

Suppose $x+B_1, x+B_2 \in \SA$.  Then
$\supp(B_1 - B_2) \subseteq \DD$ and
\begin{align*}
	\Psi(x+B_1)-\Psi(x+B_2) &=
	(T + B_1\circ T - 1) - (T + B_2\circ T - 1)
	\\ &=
	(B_1 - B_2)\circ T \fst B_1 - B_2 .
\end{align*}
So $\Psi$ is contractive.

Now we are ready to apply the well based
contraction theorem \cite[Thm~4.7]{hoevenop}.
In our case where
$\G$ is totally ordered, the dotted ordering
$\;\fst\!\!\!\cdot\;$ of \cite{hoevenop}
coincides with the usual ordering $\;\fst\;$.
There is $V \in \SA$ such that $\Psi(V) = V$.
This is what was required.

Part (b) is proved from part (a) as before:
Begin with $T = x - 1 + A$ purely deep, then
$T^{[-1]} = x + 1 + A_1$ also purely deep, from
part (a) get $V = x + B$ with $V\circ T^{[-1]} = V+1$,
so compose with $T$ on the right to get
$V = V\circ T+1$ as desired.
\end{proof}

Because they depend on Proposition~\ref{pdabel}, the following
two results are
also valid only for the well-based version of $\T$.

\begin{thm}\label{thm:abel}
{\rm\fbox{W} (General Abel)}
Let $T \in \LP$ with $\expo T = 0$.
Then there is $V \in \LP$ such that:
{\rm(i)}~If $T>x$, then $V\circ T\circ V^{[-1]} = x+1$;
{\rm(ii)}~If $T<x$, then $V\circ T\circ V^{[-1]} = x-1$.
\end{thm}
\begin{proof}
(i) First, there is $V_1$ so that
$T_1 := V_1\circ T \circ V_1^{[-1]} \sim x$.  By
Proposition~\ref{reduce}, there is $V_2$ so that
$T_2 := V_2\circ T_1 \circ V_2^{[-1]} = x + 1 + B$
with $B^\dagger \fst 1$.  By Proposition~\ref{pdabel}
there is $V_3$ so that
$V_3\circ T_2 \circ V_3^{[-1]} = x+1$.
Define $V = V_3\circ V_2\circ V_1$ to get
$V\circ T \circ V^{[-1]} = x+1$.

(ii) is similar.
\end{proof}

\begin{co}\label{realiter}
{\rm\fbox{W}}
Let $T \in \LP$ with $\expo T = 0$.
Then there exists real iteration group $\Phi(s,x)$
for $T$.
\end{co}
\begin{proof}
In case $T>x$, let $V$ be as in Theorem~\ref{thm:abel}(i),
then take $\Phi(s,x) = V^{[-1]}\circ(x+s)\circ V$.
In case $T<x$, let $V$ be as in Theorem~\ref{thm:abel}(ii),
then take $\Phi(s,x) = V^{[-1]}\circ(x-s)\circ V$.
\end{proof}

\begin{qu}
The proof as given here depends on the existence of inverses
in $\LP$. Is it possible
to demonstrate first the solution to
Abel's Equation without assuming the existence
of inverses, then use that to construct inverses?
\end{qu}

\begin{ex}\label{ex2abel}
Take the example $T = x + 1 + xe^{-x^2}$ of
Proposition~\ref{super}.  Carrying out the
iteration of Proposition~\ref{pdabel},
we get $V$ satisfying $V(T(x)) = V(x)+1$ which looks like:
\begin{align*}
	V = &\;x
	\\ & + e^{-x^2}\Big(x + e^{-2x}(x+1)e^{-1} + 
	e^{-4x}(x+2)e^{-4} + e^{-6x}(x+3)e^{-9x} +\cdots\Big)
	\\ & + e^{-2x^2}\Big(e^{-2x}(-x-4x^2-2x^3)e^{-1}
	+e^{-4x}(x-4x^2-2x^3)e^{-4}
	\\ &\qquad\qquad +e^{-6x}\big[(7x-4x^2-2x^3)e^{-9}
	+(-7-15x-10x^2-2x^3)e^{-5}\big]+\cdots\Big)
	\\ & + e^{-3x^2}\Big(e^{-2x}(-x^2+3x^3+6x^4+2x^5)e^{-1}
	+e^{-4x}(-2x^2-3x^3+4x^4+2x^5)e^{-4}
	\\ &\qquad\qquad +e^{-6x}\big[(+5x^2-13x^3+2x^4+2x^5)e^{-9}
	\\ &\qquad\qquad\qquad
	+(-x+38x^2+74x^3+44x^4+8x^5)e^{-5}\big]+\cdots\Big)
	\\ & + e^{-4x^2}\Bigg(e^{-2x}\left(\frac{5}{3}x^3+\frac{8}{3}x^4-4x^5
	-\frac{16}{3}x^6-\frac{4}{3}x^7\right)e^{-1}
	\\ &\qquad\qquad +e^{-4x}\left(-x^3+4x^4+4x^5-\frac{8}{3}x^6
	-\frac{4}{3}x^7\right)e^{-4}
	+\cdots\Bigg)
	\\ & + \cdots
\end{align*}
The support is a subgrid of order type $\omega^2$.
\end{ex}

Of course, once we have $V$ we can compute the real iteration group
$\Phi(s,x) = V^{[-1]}\circ(x+s)\circ V$.
For $s$ negative we get
$$
	\Phi(s,x) = x+s-xe^{-(x+s)^2}+\cdots,
$$
so they are not contained in a common grid (or well ordered set),
as noted before.  But since $V$ is grid-based, all of the
fractional iterates $T^{[s]}$ are also grid-based.

\subsection*{The Non-Grid Situation}
\begin{ex}\label{not_grid}
Here is an example where Abel's Equation has no grid-based solution.
\begin{equation*}
T = x + 1 + e^{\displaystyle -e^{x^2}} .
\end{equation*}
The support for $V$ where $V\circ T = V+1$ deserves
careful examination.  We will use these notations:
\begin{align*}
	\m_1 &= x^{-1},\qquad \m_1 \fst 1,\qquad \m_1 \in \G_0,
	\\
	\m_2 &= e^{-x},\qquad \m_2 \fst \m_1,\qquad \m_2 \in \G_1,
	\\
	\m_3 &= e^{-x^2},\qquad \m_3 \fst \m_2,\qquad \m_3 \in \G_1,
	\\
	\bmu &= \{\m_1,\m_2,\m_3\} \subset \G_1,
	\\
	L_k &= e^{(x+k)^2} = e^{k^2}\m_2^{-2k}\m_3^{-1},\qquad k = 0,1,2,\cdots,
	\qquad \supp L_k \subset \GRID^\ebmu \subset \G_1,
	\\
	\fa_k &= e^{-L_k},\qquad k=0,1,2,\cdots,
	\qquad \fa_k \in \G_2 ,
	\\
	\fb_k &= x\m_2^{-2k}\m_3^{-1}\fa_k,\qquad
	\ba = \bmu \cup
	\SET{\fb_k}{k=0,1,2,\cdots} \subset \G_2.
\end{align*}
Now $\ba$ is infinite, so it is not a ratio set in the usual sense.
However, writing $\g_1 \fgg \g_2$ iff
$\g_1^k \fgt \g_2$ for all $k \in \N$, we have
$$
	\m_1 \fgg \m_2 \fgg \m_3 \fgg
	\fb_0 \fgg \fb_1 \fgg \fb_2 \fgg \cdots .
$$
The semigroup
generated by $\ba$ is contained in $\G_2$, is
well ordered, and has order type $\omega^\omega$.
Probably the solution $V$ of Abel's Equation also
has support of order type $\omega^\omega$, but to
prove it we would have to verify that many
terms are not eliminated by cancellation.

Computations follow.  When I write
$\o$ and $\O$, the omitted terms all belong to $\G_2$.
\begin{align*}
	T &= x + 1 + \fa_0,
	\\
	T^2 &= x^2 + 2x + 1 + 2x\fa_0 + \O(\fa_0) ,
	\\
	\m_2^{-1}\circ T &= e^{T} = e^{x+1+\fa_0} =
	e \m_2^{-1}e^{\fa_0} = e \m_2^{-1}\big(1+\fa_0+\o(\fa_0)\big)
	\\ &= e \m_2^{-1} + e \m_2^{-1}\fa_0 + \o(\m_2^{-1}\fa_0) ,
	\\
	\m_2^{-2k}\circ T &= e^{2k}\m_2^{-2k} + 2ke^{2k}\m_2^{-2k}\fa_0
	+ \o(\m_2^{-2k}\fa_0) ,
	\\
	\m_3^{-1}\circ T &= e^{T^2} = e^{x^2 + 2x + 1 + 2x\fa_0 + \O(\fa_0)}
	= e \m_2^{-2}\m_3^{-1}e^{2x\fa_0 + \O(\fa_0)}
	\\ &= e \m_2^{-2}\m_3^{-1}\big(1 + 2x\fa_0 + \O(\fa_0)\big)
	\\ &= e \m_2^{-2}\m_3^{-1} + 2e x\m_2^{-2}\m_3^{-1}\fa_0
	+ \O(\m_2^{-2}\m_3^{-1}\fa_0) ,
	\\
	L_k\circ T &= e^{k^2+2k+1}\m_2^{-2k-2}\m_3^{-1} +
	2e^{k^2+2k+1}x\m_2^{-2k-2}\m_3^{-1}\fa_0
	+ \O(\m_2^{-2k-2}\m_3^{-1}\fa_0)
	\\
	\fa_k\circ T &= e^{-L_{k+1}}
	e^{-2e^{(k+1)^2}x\m_2^{-2k-2}\m_3^{-1}\fa_0
	+ \O(\m_2^{-2k-2}\m_3^{-1}\fa_0)}
	\\ &= \fa_{k+1}\big(1-2e^{(k+1)^2}x\m_2^{-2k-2}\m_3^{-1}\fa_0
	+ \O(\m_2^{-2k-2}\m_3^{-1}\fa_0)\big) 
	\\ &= \fa_{k+1} -2e^{(k+1)^2}x\m_2^{-2k-2}\m_3^{-1}\fa_0\fa_{k+1}
	+ \O(\m_2^{-2k-2}\m_3^{-1}\fa_0\fa_{k+1})
	\\ &= \fa_{k+1} -2e^{(k+1)^2}\fa_0\fb_{k+1} + \o(\fa_0\fb_{k+1})
	\\
	\fb_k\circ T &= (x\m_2^{-2k}\m_3^{-1}\fa_k)\circ T
	= \fb_{k+1} + \o(\fb_{k+1}) .
\end{align*}
The solution $V$ of Abel's Equation $V \circ T = V + 1$ is
$$
	V = x + 1 + \fa_0 + \fa_0\circ T + \fa_0\circ T^{[2]} + 
	\fa_0\circ T^{[3]} + \cdots .
$$
Without considering cancellation, we would expect
that its support still
has order type $\omega^\omega$.  Even without trying
to account for cancellation, we know that $\supp V$ contains
$\{\fa_0, \fa_1, \fa_2, \cdots\}$.  The logarithms
$L_k$ are linearly independent, so the group generated
by $\SET{\fa_k}{k \in \N}$ is not finitely generated,
and thus $\supp V$ is not a subgrid. \qed

More computation in this example yields:
$V^{[-1]} = x - 1 - \fa_{-1} +\cdots$ and
$$
	T^{[1/2]} = V^{[-1]} \circ \left(x+\frac{1}{2}\right) \circ V
	= x + \frac{1}{2} + \fa_0 - \fa_{1/2} + \cdots
$$
not grid-based.  We used notation
$\fa_k = \exp\big({-}\exp((x+k)^2)\big)$ for
$k=-1$ and $1/2$.
\end{ex}

Consider the proof of Proposition~\ref{pdabel}.
How much can be done in the grid-based version?
Assume grid-based
$T = x + 1 + A$, $A \in \T$, $A \fst 1$, $A^\dagger \fgt 1$.
Write $\m = \mag A$.  Consider a ratio set $\bmu$ such that
$\supp A \subseteq \m\GRID^{\ebmu,\0}$.
There is \cite[\Wcompomono]{edgarw} a ``$T$-composition addendum''
$\ba$ for $\bmu$ such that:

(i) if $\fa \in \GRID^{\ebmu,\0}$, then
$\supp(\fa\circ T) \subseteq \GRID^{\ba,\0}$;

(ii) if $\fa \fst^\ebmu \fb$, then
$\fa\circ T \fst^\ba \fb\circ T$.

\noindent
But this is not enough to carry out the contraction argument.
We need a ``hereditary $T$-composition addendum''
$\ba \supseteq \bmu$ such that:

(i) if $\fa \in \GRID^{\ba,\0}$, then
$\supp(\fa\circ T) \subseteq \GRID^{\ba,\0}$;

(ii) if $\fa \fst^\ba \fb$, then
$\fa\circ T \fst^\ba \fb\circ T$.

\noindent
For some deep $T$ there is such an addendum, but not for others.
If there is, then a grid-based version of the contraction argument
of Proposition~\ref{pdabel} works.
Or (for purely deep $T$) we can write
$$
	V = x + 1 + A + A\circ T + A\circ T^{[2]} +
	A\circ T^{[3]} + \cdots
$$
with $A \fgt^\ba A\circ T \fgt^\ba  A\circ T^{[2]} \fgt^\ba  \cdots$
to insure grid-based convergence
in the asymptotic topology.

For example, if $\bmu = \{x^{-1},e^{-x},e^{-x^2}\}$
and $\supp A \subseteq \GRID^{\ebmu,\0}$,
then $\ba = \{x^{-1},e^{-x},x^2e^{-x^2}\}$
is a hereditary $T$-composition addendum.
This insures that the iteration used in Example~\ref{ex2abel}
provides a grid-based solution $V$.

\section{Uniqueness}
In what sense is $T^{[s]}$ unique?
This question is related to the question of
commutativity for composition.

\begin{pr}\label{commutex1}
Let $V \in \LP$.  If $V(x+1) = V+1$, then there is $c \in \R$
with $V = x+c$.
\end{pr}
\begin{proof}
By \cite[\Cmvtiii]{edgarc}, $1 = (V(x+1)-V(x))/1 = V'\circ S$ for some
$S \in \LP$.
Compose on the right with $S^{[-1]}$ to get $1=V'$,
So $V = x+c$ as required.
\end{proof}

\begin{co}\label{abelunique}
Let $T \in \LP$, $T > x$.  The solution $V \in \LP$
of $V\circ T = V+1$ is unique up to a constant addend.
\end{co}
\begin{proof}
Suppose $V\circ T = V+1$ and $U\circ T = U+1$.
Then $V^{[-1]}\circ(x+1)\circ V = U^{[-1]}\circ(x+1)\circ U$
and $(U\circ V^{[-1]})\circ(x+1) = (x+1)\circ(U\circ V^{[-1]})$.
By Proposition~\ref{commutex1} there is $c \in \T$
with $U\circ V^{[-1]} = x+c$ so that $U = V+c$.
\end{proof}

\begin{no}\label{defiter}
Let $T \in \LP$, $s \in \R$.  If $T>x$, define
$T^{[s]} = V^{[-1]}\circ(x+s)\circ V$, where $V$ is a solution
of Abel's Equation $V(T) = V+1$.  If $T<x$, define
$T^{[s]} = V^{[-1]}\circ(x-s)\circ V$, where $V$ is a solution
of Abel's Equation $V(T) = V-1$.  The transseries
$T^{[s]}$ is independent of the choice of solution $V$.
\end{no}

Note: Even if $T$ is grid-based, it could happen that
$T^{[s]}$ is not.

\begin{pr}\label{power}
Let $A,B \in \LP$, $B \ne x$.  If $A\circ B = B \circ A$, then there
is $s \in \R$ with $B^{[s]} = A$.
\end{pr}
\begin{proof}
We do the case $B>x$; the case $B<x$ is similar.
Let $V \in \LP$ solve Abel's Equation
for $B$, so that $B^{[s]} = V^{[-1]}\circ(x+s)\circ V$ for
$s \in \R$.  Then $V^{[-1]}\circ(x+1)\circ V \circ A =
A \circ V^{[-1]}\circ(x+1)\circ V$.  Compose with $V$ on
the left and $V^{[-1]}$ on the right to get
$(x+1)\circ (V\circ A \circ V^{[-1]}) =
(V \circ A \circ V^{[-1]})\circ (x+1)$.
By Proposition~\ref{commutex1}, there is $s \in \R$
with $V\circ A \circ V^{[-1]} = x+s$.  So
$A = V^{[-1]}\circ(x+s)\circ V = B^{[s]}$.

If $A,B$ are grid-based, perhaps $B^{[s]}$ is in general
not grid-based.  But since we conclude $B^{[s]} = A$, then
at least for this particular $s$ it happens to be
grid-based.
\end{proof}

\begin{ex}
Let $\theta \takes \R \to \R$ satisfy $\theta(1)=1$ and
$\theta(s+t) = \theta(s)+\theta(t)$ for all $s,t$.
By the Axiom of Choice, there is
such a map $\theta$ other than the identity function
$\theta(s) = s$.  (This strange $\theta$ is everywhere
discontinuous, non-measurable, unbounded on every interval.)
Let $T \in \LP$.
Then $\Phi(s,x) = T^{[\theta(s)]}$ is a real
iteration group for $T$.
\end{ex}

Here is a way to rule out such strange cases.

\begin{pr}
Let $T \in \LP$, $T > x$, and let $\Phi(s,x)$ be a real iteration
group for $T$.  Assume $\Phi(s,x)>x$ for all $s>0$.
Then $\Phi(s,x) = T^{[s]}$ as in~\ref{defiter}.
\end{pr}
\begin{proof}
Since $\Phi(s,x) > x$ for $s>0$, we may deduce that $s_1<s_1$ implies
$\Phi(s_1,x)<\Phi(s_2,x)$.  Also $\Phi(1,x) = T$,
so we may deduce $\Phi(s,x) = T^{[s]}$ for all rational $s$.
Fix an irrational $s$.
Since $\Phi(s,x)\circ\Phi(1,x) = \Phi(s+1,x)$, we know
that $\Phi(s,x)$ commutes with $T$, so by
Proposition~\ref{power}, $\Phi(s,x) = T^{[t]}$
for some $t$.  But the only $t$ satisfying
$T^{[s_1]} < T^{[t]} < T^{[s_2]}$ for all rationals
$s_1,s_2$ with $s_1 < s < s_2$ is $t=s$ itself.
\end{proof}

Similarly:
Let $T \in \LP$, $T < x$, and let $\Phi(s,x)$ be a real iteration
group for $T$.  Assume $\Phi(s,x)<x$ for all $s>0$.
Then $\Phi(s,x) = T^{[s]}$ as in~\ref{defiter}.

\section{Julia Example}\label{sec:julia}
As an example we will consider fractional iterates for
the function $M(x) = x^2+c$ near $x=+\infty$.  Of course,
integer iterates of this function are used for construction
of Julia sets or the Mandelbrot set.  For the theory of
real transseries to be applicable, we must restrict to real
values $c$.  But once we have nice formulas, they can
be investigated for general complex $c$.
In the case $c=-2$ there is a closed form known,
$M^{[s]} = 2\cosh(2^s\acosh(x/2))$.
[Of course, $x^2-2 = 2\cosh(2 \acosh(x/2))$
is essentially the
double-angle formula for cosines.]  And of course
in the case $c=0$ the closed form is $M^{[s]} = x^{2^s}$.
For other values of $c$, no closed form is known, and
it is likely that there is none (but that must be
explained).

So, let $c$ be a fixed real number, and write
$M(x) = x^2+c$.  Use ratio set
$\bmu = \{\mu_0,\mu_1,\mu_2,\mu_3\}$,
\begin{align*}
	\mu_0 &= \frac{1}{\log x},\quad
	\mu_1 = x^{-1},\quad
	\mu_2 = e^{-x},\quad
	\mu_3 = e^{-e^x} .
\end{align*}

Begin with $M(x) = x^2+c$.  Then
\begin{equation*}
	M_1 := \log\circ M \circ \exp = \log\big(e^{2x}+c\big)
	= \log\big(e^{2x}(1+ce^{-2x})\big)
	= 2x - \sum_{j=1}^\infty \frac{(-1)^jc^j\mu_2^{2j}}{j} .
\end{equation*}
The series in powers of $\mu_2$.  Next,
\begin{align*}
	M_2 := \log\circ M_1 \circ \exp
	= \log\left(2e^x-\sum_{j=1}^\infty\frac{(-1)^j c^j\mu_3^{2j}}{j}\right)
	= \log\left(2e^x\left(1-\sum_{j=1}^\infty
	\frac{(-1)^jc^{j}\mu_2\mu_3^{2j}}{2j}\right)\right) .
\end{align*}
Writing $A$ for the series (in powers of $\mu_2, \mu_3$),
\begin{align*}
	M_2 = x + \log 2 - \sum_{j=1}^\infty\frac{A^j}{j}
	= x + \log 2 +\frac{c}{2}\,\mu_2\mu_3^2 - \frac{c^2}{4}\,\mu_2\mu_3^4
	-\frac{c^2}{8}\,\mu_2^2\mu_3^4 + \O(\mu_2\mu_3^6).
\end{align*}
The $\O$ term represents $\mu_2\mu_3^6$ times
a series in $\mu_2, \mu_3$ with nonnegative exponents.

Note that $M_2$ is \emph{deep} in the sense of Definition~\ref{de:deep},
so the iterates will be computed using Abel's Equation.  The
solution $V$ of the Abel equation
$$
	V \circ M_2 = V + \log 2
$$
is found by iteration $V_0 = x$,
$V_{n+1} = V_n\circ M_2 - \log 2$.  The result is
\begin{equation*}
	V = x +\frac{c}{2}\,\mu_2\mu_3^2+\frac{c-c^2}{4}\,\mu_2\mu_3^4
	- \frac{c^2}{8}\,\mu_2^2\mu_3^4 -\frac{c^2}{2}\,\mu_2\mu_3^6
	-\frac{c^2}{8}\,\mu_2^2\mu_3^6 + \O(\mu_2\mu_3^8) .
\end{equation*}
The $\O$ is a series in $\mu_2,\mu_3$.  For $c=-2$, closed form is:
$$
	V = \log\acosh\frac{1}{2}\,e^{e^x} .
$$
The inverse is computed as in \cite[\Cinversei]{edgarc}:
\begin{align*}
	V^{[-1]} &= x -\frac{c}{2}\, \mu_2\mu_3^2
	+\frac{c-c^2}{4}\,\mu_2\mu_3^4
	-\frac{c^2}{8}\,\mu_2^2\mu_3^4
	+ \frac{c^3-3c^2}{6}\,\mu_2\mu_3^6
	\\ &\qquad
	+ \frac{c^3-c^2}{8}\mu_2^2\mu_3^6
	+ \frac{c^3}{24}\mu_2^3\mu_3^6
	+ \O(\mu_2\mu_3^8) .
\end{align*}
The iteration group is then a computation; for any real $s$,
\begin{equation*}
	M_2^{[s]} = V^{[-1]}\big(V(x) + s \log 2\big) .
\end{equation*}
For a fixed $s$, augment our ratio set with
$$
	\mu_4 = x^{-2^s},\quad
	\mu_5 = e^{-2^sx},\quad
	\mu_6 = e^{-2^se^x} .
$$
Then
\begin{align*}
	M_2^{[s]} = x + s\log 2
	+\frac{c}{2}\,\mu_2\mu_3^2
	- 2^{-1-s}c\mu_2\mu_6^2
	+\frac{c-c^2}{4}\,\mu_2\mu_3^4
	-\frac{c^2}{8}\,\mu_2^2\mu_3^4 &
	\\
	+ \frac{c^2}{2}\, \mu_2\mu_3^2\mu_6^2
	+ 2^{-2-s} c^2\mu_2^2\mu_3^2\mu_6^2
	+\O(\mu_2\mu_3^6+\mu_2\mu_6^4) &.
\end{align*}
Which of the two terms in the $\O$ is larger depends on the value of $s$.
The relative sizes of the terms shown also depend on the value of $s$.
If $s>0$, then $\mu_3 \fgt \mu_6$, so
$M_2^{[s]} = x + s\log 2 +(c/2)\mu_2\mu_3^2 + \cdots$.  If $s<0$, then
$\mu_3 \fst \mu_6$, so
$M_2^{[s]} = x + s\log 2 -2^{-1-s}c \mu_2\mu_6^2 + \cdots$.

Continue:
\begin{align*}
	M_1^{[s]} = \exp\circ M_2^{[s]}\circ\log
	= 2^s x
	+ 2^{-1+s}c \mu_2^2
	-\frac{c}{2}\, \mu_5^2
	+2^{-2+s}(c-c^2)\mu_2^4
	+ 2^{-1+s}c^2\mu_2^2\mu_5^2 &
	\\ 
	+ \big(2^{-2+s}(c^2-c^3) - 2^{-2+2s}c^3\,\big) \mu_2^4\mu_5^2
	-\frac{c^2+c}{4}\mu_5^4
	+ \O(\mu_2^6+\mu_5^6) &.
\end{align*}
This is a series in $\mu_1, \mu_2, \mu_5$.  The coefficients
involve rational numbers and powers of $2^s$.  I do not
know if $\mu_1$ actually appears: up to this point,
all terms with $\mu_1$ cancel.  Next,
\begin{align*}
	M^{[s]} = \exp\circ\; M_1^{[s]}\circ\log =
	x^{2^s}\Big(
	1 + 2^{-1+s}c\mu_1^2 -\frac{c}{2}\,\mu_4^2
	+\big(2^{-2+s}(c-c^2) + 2^{-3+2s}c^2\big)\mu_1^4 &
	\\ 
	+2^{-2+s}c^2 \mu_1^2\mu_4^2 + \O(\mu_1^6+\mu_4^4)
	\Big) &.
\end{align*}
This is a series in $\mu_0, \mu_1, \mu_4$, but I do not know if
$\mu_0 = 1/\log x$ actually appears.
The relative size of the terms depends on the value of $s$.

Let us substitute a few example values of $s$ into this transseries:
\begin{align*}
	M^{[1]} &= x^2\big(1+cx^{-2}+\O(x^{-6})\big) = x^2+c+\O(x^{-4})
	\\
	M^{[-1]} &= x^{1/2}\left(1-\frac{c}{2}\,x^{-1}+\O(x^{-2})\right) =
	x^{1/2}-\frac{c}{2}\,x^{-1/2} +\O(x^{-3/2})
	\\
	M^{[1/2]} &= x^{\sqrt{2}}
	+2^{-1/2}c x^{\sqrt{2}-2}
	-\frac{c}{2}\,x^{-\sqrt{2}}
	+\left(2^{-3/2}(c-c^2)+\frac{c^2}{4}\right)x^{\sqrt{2}-4}
	\\ &\qquad\qquad
	+2^{-3/2}c^2x^{-2-\sqrt{2}}
	+\O\left(x^{-3\sqrt{2}}\right) .
\end{align*}
In case $c=-2$ we have:
\begin{equation*}
	M^{[1/2]} = x^{\sqrt{2}}
	- \sqrt{2}\,x^{\sqrt{2}-2} + x^{-\sqrt{2}}
	+\left(1-\frac{3}{\sqrt{2}}\right)x^{\sqrt{2}-4}
	+ \sqrt{2}\,x^{-2-\sqrt{2}}
	+\O\left(x^{-3\sqrt{2}}\right) ,
\end{equation*}
which does match the transseries for the closed form
$$
	M^{[1/2]} = 2\cosh\left(\sqrt{2}\,\acosh\frac{x}{2}\right) .
$$

\begin{figure*}[htb] 
   \centering
   \includegraphics[width=4.65in]{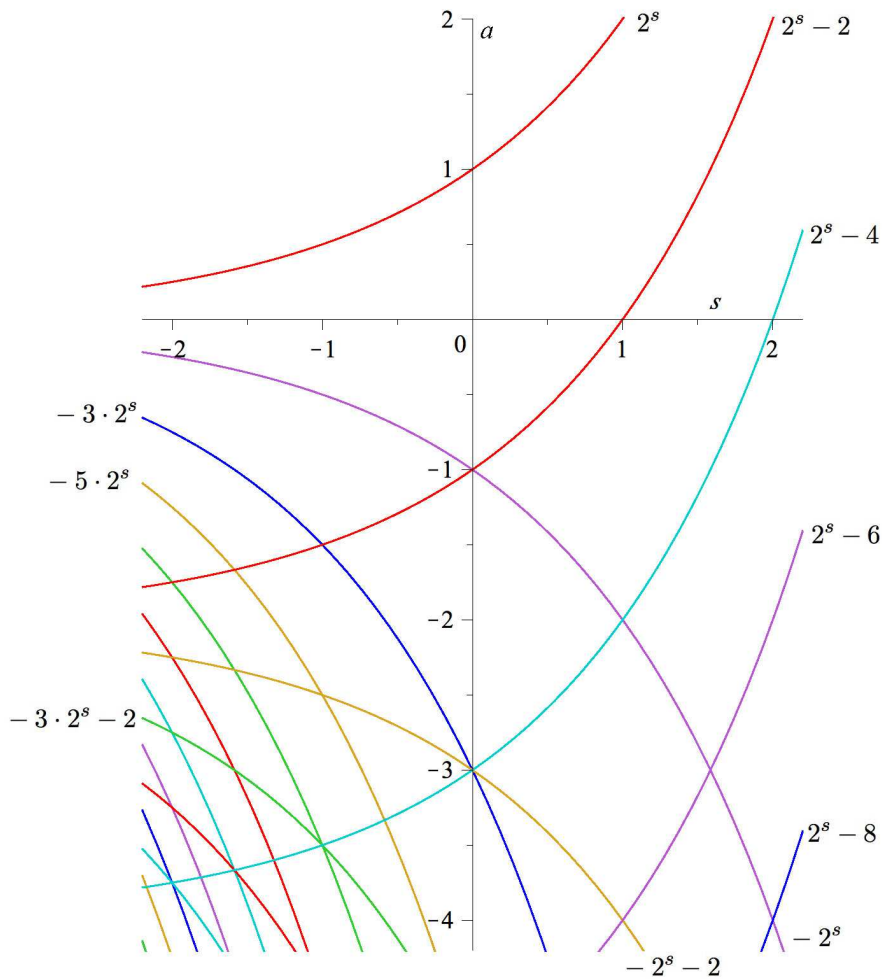} 
   \caption{$M^{[s]}$ supported by monomials $x^a$}
   \label{fig:supp}
\end{figure*}

\noindent
Figure~\ref{fig:supp} illustrates the the support of $M^{[s]}$
depending on $s$.  The support of $M^{[s]}$
consists of certain monomials
of the form $x^a$, where points $(s,a)$ are shown in
the figure.  I have assumed that logarithmic factors are,
indeed, missing.  Or perhaps we could say: any monomials with
logarithmic factors differ only infinitesimally from the terms
shown, so even if they do exist,
they make no difference in the picture.

\end{document}